\documentclass[11pt]{amsart}

\usepackage[all]{xy}
\usepackage{amsmath}
\usepackage{amsfonts,mathrsfs,bbding}
\usepackage{amssymb}
\usepackage{amscd}
\usepackage{amsthm}
\usepackage[titletoc]{appendix}
\usepackage{latexsym}
\usepackage{graphicx}
\usepackage{epstopdf}
\usepackage{amsrefs}
\usepackage{mathscinet}
\usepackage{epsfig}
\usepackage{amsthm}
\usepackage{latexsym}
\usepackage[latin1]{inputenc}
\usepackage{verbatim}
\usepackage{pb-diagram}
\usepackage[all]{xy}
\usepackage{lipsum}
\AppendGraphicsExtensions{.gif}

\theoremstyle{definition}
\newtheorem{definition}{Definition}[section]

\newtheorem{example}{Example}[section]

\newtheorem{remark}{Remark}[section]
\newtheorem{remarks}{Remarks}[section]
\newtheorem{observation}{Observation}[section]

\theoremstyle{plain}
\newtheorem{theorem}{Theorem}[section]
\newtheorem{corollary}{Corollary}[theorem]
\newtheorem{lemma}[theorem]{Lemma}
\newtheorem{proposition}[theorem]{Proposition}
\def\limind{\mathop{\oalign{lim\cr\hidewidth$\longrightarrow$\hidewidth\cr}}} 
\def\limproj{\mathop{\oalign{lim\cr\hidewidth$\longleftarrow$\hidewidth\cr}}} 

\newcommand{\compcent}[1]{\vcenter{\hbox{$#1\circ$}}}
\newcommand{\ip}[2]{\left\langle #1 , #2 \right\rangle}    
\newcommand{\comp}{\mathbin{\mathchoice
{\compcent\scriptstyle}{\compcent\scriptstyle}
{\compcent\scriptscriptstyle}{\compcent\scriptscriptstyle}}} 

\newcommand{\op}{\operatorname }
\newcommand{\cali}{\mathscr }

\newcommand{\m}{\mathcal }
\newcommand{\ba}{\overline}

\title{$C^*$-subproduct and product systems}

\author{Remus Floricel}
\address{University of Regina, Department of Mathematics, Regina, SK, Canada}
\email{Remus.Floricel@uregina.ca}
\author{Brian Ketelboeter} 
\address{University of Regina, Department of Mathematics, Regina, SK, Canada}
\email{ketelbob@uregina.ca}
\date{\today}
\begin{document}

\maketitle
\begin{abstract}
We introduce and study two-parameter subproduct and product systems of $C^*$-algebras as the
operator-algebraic analogues of, and in relation to, Tsirelson's two-parameter product systems of Hilbert spaces. Using several inductive limit techniques, we show that (i) any $C^*$-subproduct system can be dilated to a $C^*$-product system; and (ii) any $C^*$-subproduct system that admis a unit, i.e., a co-multiplicative family of projections, can be assembled into a $C^*$-algebra, which comes equipped with a one-parameter family of comultiplication-like homomorphisms. We also introduce and discuss co-units of $C^*$-subproduct systems, consisting of co-multiplicative families of states, and show that they correspond to idempotent states of the associated $C^*$-algebras. We then use the GNS construction to obtain Tsirelson subproduct  systems of Hilbert spaces from co-units, and describe the relationship between the dilation of a $C^*$-suproduct system and the dilation of the Tsirelson subproduct system of Hilbert spaces associated with a co-unit. All these results are illustrated concretely at the level of $C^*$-subproduct systems of commutative $C^*$-algebras.
\end{abstract}

\section{Introduction}
Our primary goal, in this paper, is to initiate the study of two-parameter subproduct and product systems of $C^*$-algebras, henceforth referred to as $C^*$-subproduct and product systems, building on the following: (i) Arveson's theory of one-parameter product systems of Hilbert spaces \cite{Arveson89, Arveson97, Arveson-book}, including the work of B.V.R. Bhat and M. Mukherjee \cite{Bhat-M}, and that of O. Shalit and B. Solel \cite{Shalit-Solel}, on subproduct systems of Hilbert spaces, as well as  the work of B.V.R. Bhat and M. Skeide on product systems of Hilbert modules \cite{Bhat-Skeide, Ske}; and (ii)
Tsirelson's theory of two-parameter product systems of Hilbert spaces \cite{Tsi03, Tsi04} (see also \cite{Liebscher}).

One-parameter product systems of Hilbert spaces, commonly known as product systems, or Arveson product systems, are main objects of study in the theory of noncommutative dynamics \cite{Arveson-book}. They were introduced by W. Arveson in \cite{Arveson89} as an effective tool for classifying $E_0$-semigroups up to cocyle conjugacy, thus opening a new direction of research, complementary to that initiated by R.T. Powers in \cite{Powers88} (see also \cite{Pow99, Pow03}). At algebraic level, Arveson product systems of Hilbert spaces can simply be defined as pairs $\left(\{E_{t}\}_{0<t}, \{V_{s,t}\}_{0<s,\, t}\right)$ consisting of a family $\{E_t\}_{t>0}$ of Hilbert spaces $E_t$ and a family $\{V_{s,t}\}_{s,t>0}$ of unitary operators $V_{s,t}:E_{s+t}\to E_s\otimes E_t$ that satisfy the co-associativity law \begin{eqnarray}\label{Feb6}(1_{E_r}\otimes V_{s,t})V_{r,s+t}=(V_{r,s}\otimes 1_{E_t})V_{r+s,t},\end{eqnarray} for all $r,\,s,\,t>0$. We emphasize, however, that Arveson's original definition of a product system also requires that all Hilbert spaces $E_t$ be separable and that the bundle $\{E_t\}_{t>0}$ be endowed with a certain Borel measurable structure that makes all operations measurable, in which case the system is referred to as a ``measurable Arveson system of Hilbert spaces''.  Nevertheless, since the measurable structure can be dropped whenever it comes to the subject of classifying product systems up to isomorphism  \cite{Liebscher}, which incidentally is one of the motivating reasons of this work, we do not impose these additional conditions on any of the systems considered in this article.

The study of Arveson product systems can be enriched by taking into account two-parameter product systems of Hilbert spaces. Such systems were introduced by B. Tsirelson in \cite{Tsi03} with the primary purpose of constructing a continuum of mutually non-isomorphic Arveson product systems, a task accomplished by the use of a wide range of probabilistic techniques (see also \cite{Tsi04, Liebscher}). In our paper, two-parameter product systems are referred to as ``Tsirelson product systems of Hilbert spaces", and defined as pairs $\left(\{H_{s,t}\}_{0<s<t}, \{U_{r,s,t}\}_{0<r<s<t}\right)$ consisting of a family  $\{H_{s,t}\}_{0<s<t}$ of Hilbert spaces $H_{s,t}$ and a family $\{U_{r,s,t}\}_{0<r<s<t}$ 
 of unitary operators $U_{r,s,t}: H_{r,t} \to H_{r,s}\otimes H_{s,t}$ that satisfy the co-associativity law 
 \begin{eqnarray}\label{Jan24cc}
 \left(1_{H_{r,s}}\otimes U_{s,t,u}\right)U_{r,s,u}=\left(U_{r,s,t}\otimes 1_{H_{t,u}}  \right)U_{r,t,u},
 \end{eqnarray} 
for all positive real numbers $0<r < s < t < u$. Strictly speaking, the notion of Tsirelson product system of Hilbert spaces defined above corresponds to Tsirelson's notion of local continuous product of Hilbert spaces (local CP HS) in \cite{Tsi03}, obtained by excluding unbounded intervals form the definition of a CP HS.

The concept of subproduct system also plays a central role in the theory of noncommutative dynamics. Its introduction was mainly motivated by the study of quantum dynamical semigroups, also known as $CP$-semigroups \cite{Arveson-book}. It was shown by B.V.R. Bhat in \cite{Bhat} (see also \cite{Bhat2, Arveson-book}) that any unital quantum dynamical semigroup can be dilated to an $E_0$-semigroup. This procedure, known as the Bhat dilation of a quantum dynamical semigroup, can be regarded as an operator algebraic counterpart of the classical  Kolmogorov extension theorem. As in the case of $E_0$-semigroups, W. Arveson showed in \cite{Arveson97} (see also \cite{Markiewicz}) that quantum dynamical semigroups can also be described in terms of a certain one-parameter family of Hilbert spaces and an associated family of isometric operators. This, together with other important developments (see \cite{Shalit-Solel,Bhat-M, Shalit-Skeide} for a comprehensive description), led to the emergence of the formal concept of subproduct system, which was concurrently introduced by O. Shalit and B. Solel in \cite{Shalit-Solel} under the name ``subproduct system", and by B.V.R. Bhat and M. Mukherjee in \cite{Bhat-M} under the name ``inclusion system". Although we do not intend to address this concept in detail in this article, in order to distinguish it from other similar notions, we simply call it ``Arveson subproduct system of Hilbert spaces", and by this we mean a pair $\left(\{E_{t}\}_{0<t}, \{V_{s,t}\}_{0<s,\, t}\right)$ consisting of a family $\{E_t\}_{t>0}$ of Hilbert spaces $E_t$ and a family $\{V_{s,t}\}_{s,t>0}$ of isometric operators $V_{s,t}:E_{s+t}\to E_s\otimes E_t$ that satisfy the co-associativity law (\ref{Feb6}).

Closely related to all the concepts discussed so far is the concept of ``Tsirelson subproduct system of Hilbert spaces", which is defined as a pair $\left(\{H_{s,t}\}_{0<s<t}, \{U_{r,s,t}\}_{0<r<s<t}\right)$, consisting of a family of Hilbert spaces $\{H_{s,t}\}_{0<s<t}$ and a family $\{U_{r,s,t}\}_{0<r<s<t}$ of isometric operators $U_{r,s,t}: H_{r,t} \to H_{r,s}\otimes H_{s,t}$ that satisfy the co-associativity law (\ref{Jan24cc}). While Tsirelson subproduct systems of Hilbert spaces do not appear to have been properly considered and studied in the past, with the notable exception of \cite{Gurevich} in which they were discussed in a discrete setting, they can certainly have a strong impact in the study of Tsirelson product systems in general, and Arveson product systems in particular. As a matter of fact, any Tsirelson subproduct system of Hilbert spaces can be transformed into a Tsirelson product system of Hilbert spaces through a Bhat-Mukherjee type dilation procedure (see Remark \ref{monster}).
 
We aim, in this article, to complement and enrich the study of the systems of Hilbert spaces discussed above by introducing their $C^*$-algebraic versions, namely subproduct and product systems of $C^*$-algebras. We focus primarily on two-parameter subproduct and product systems of $C^*$-algebras, which are conceptually analogous to Tsirelson subproduct and product systems of Hilbert spaces, noting that any one-parameter subproduct or product system of $C^*$-algebras can be easily transformed into a two-parameter system (see Remark \ref{Feb21-22}). 

Our current investigation has been strongly motivated by the natural occurrence of two-parameter $C^*$-subproduct and product systems in the study of Arveson subproduct and product systems of Hilbert spaces, as highlighted in Examples \ref{exa1} and \ref{remus} of this article, as well as by their significant connection to the theory of $C^*$-bialgebras. We also note that some fundamental results from the theory of Arveson subproduct and product systems of Hilbert spaces have been generalized to the level of tensor categories in \cite{GLS} by using full comonoidal system over cancellative monoids as a direct equivalent, but these generalizations do not cover the case of two-parameter $C^*$-subproduct and product systems discussed in this article. 

This reminder of this paper is structured as follows. In Section \ref{3b100}, we introduce the main concepts of our study, namely the general class of tensorial $C^*$-systems (see Definition \ref{3b1}), and its subclasses of $C^*$-subproduct systems and $C^*$-product systems (see Definition \ref{May12022}), as well as the notion of unit of a tensorial $C^*$-system. A variety of examples are also discussed, ranging from $C^*$-subproduct systems of matrix algebras (Example \ref{exa1}) and tensorial $C^*$-systems of commutative $C^*$-algebras (Example \ref{examp1}) to tensorial $C^*$-systems of reduced group $C^*$-algebras (Example \ref{cucuu}) and $C^*$-product systems obtained from measurable Arveson product systems of Hilbert spaces (Example \ref{remus}). 

In Section \ref{ch3.1}, we exploit B.V.R Bhat and M. Mukherjee's inductive limit technique from \cite{Bhat-M}, which in turn is based on \cite[Section 4]{Bhat-Skeide}, and extend it to the two-parameter $C^*$-algebraic setting by showing, in Theorem \ref{star-isomorphism theorem}, that any $C^*$-subproduct system $\cali{A}=\left(\{\m{A}_{s,t}\}_{0<s<t}, \{\Delta_{r,s,t}\}_{0<r<s<t}\right)$ can be dilated to a $C^*$-product system $\cali{A}^\sharp=\left(\{\m{A}_{s,t}^\sharp\}_{0<s<t},\,\{\Delta_{r,s,t}^\sharp\}_{0<r<s<t}\right)$, called the inductive dilation of $\cali{A}$. Each $C^*$-algebra $\m{A}_{s,t}^\sharp$ of the system $\cali{A}^\sharp$ is constructed as the inductive limit of the inductive family of tensor products $\m{A}_I=\m{A}_{\iota_0, \iota_1}\otimes \m{A}_{\iota_1,\iota_2}\otimes \dots\otimes \m{A}_{\iota_m,\iota_{m+1}}$, indexed over the partially ordered set $\cali{P}_{s,t}$ of all finite partitions $I=\{s=\iota_0<\iota_1<\iota_2<\,\dots<\iota_m<\iota_{m+1}=t\}$
of the interval $[s,t]$, with embeddings obtained from the co-multiplication $ \{\Delta_{r,s,t}\}_{0<r<s<t}$ through certain tensorizations. 

Section \ref{sec4.1} is devoted to the construction of the $C^*$-algebra $C^*_{\{ p_{s,t}\}}(\cali{A})$ of a $C^*$-subproduct system $\cali{A}=\left(\{\m{A}_{s,t}\}_{0<s<t}, \{\Delta_{r,s,t}\}_{0<r<s<t}\right)$ with respect to a unit $\{p_{s,t}\}_{0<s<t}$. This construction is also carried out by an inductive limit procedure taking into account the  inductive system of $C^*$-algebras $\m{A}_I$, defined as in Section \ref{ch3.1}, but indexed this time over the partially ordered set of all finite ordered subsets $I$ of the interval $(0,\infty)$ (see Proposition \ref{lemma inductive limit}). 

In Section \ref{sec5+.1}, we show that the $C^*$-algebra $C^*_{\{ p_{s,t}\}}(\cali{A})$ admits a one-parameter family of *-homomorphisms $\Delta_s: C^*_{\{ p_{s,t}\}}(\cali{A})\to C^*_{\{ p_{s,t}\}}(\cali{A})\otimes C^*_{\{ p_{s,t}\}}(\cali{A})$, $s>0$, which satisfies the deformed co-associativity law $(\Delta_r\otimes\op{id})\Delta_s=(\op{id}\otimes\Delta_s)\Delta_r,$ for all $0<r<s$.
The family $\{\Delta_s\}_{s>0}$ is constructed in two steps. First of all, we show, by adapting to our setting some techniques from \cite[Section 5]{Bhat-Skeide} (see also  \cite{GLS}), that the $C^*$-algebras $\m{A}_{s,t}^\sharp$ of the inductive dilation $\cali{A}^\sharp$ of $\cali{A}$ can be assembled into an inductive system with respect to the unit dilation $\{p_{s,t}^\sharp\}_{0<s<t}$ of $\{p_{s,t}\}_{0<s<t}$ (Proposition \ref{Nov06}). We then deduce that the $C^*$-algebra  $C^*_{\{ p_{s,t}\}}(\cali{A})$ is *-isomorphic to the inductive limit $\m{A}^\diamond_{\{ p_{s,t}\}}$ of this inductive system (Theorem \ref{March24}). After that, we construct a family $\{\iota_s\}_{s>0}$ of *-monomorphism $\iota_s:\m{A}^{\diamond}_{\{p_{s,t}\}}\to  C^*_{\{p_{s,t}\}}(\cali{A})\otimes C^*_{\{p_{s,t}\}}(\cali{A})$ that is compatible with both inductive limit structures. 
These identifications allow us to construct the one-parameter co-multiplication $\{\Delta_s\}_{s>0}$ in Theorem \ref{April2i}.

In Section \ref{sec5.1}, we study co-units of a $C^*$-subproduct system $\cali{A}$, i.e., families of states  $\{\varphi_{s,t}\}_{0<s<t}$ that are invariant with respect to the co-multiplication of the system. For this, we introduce the concept of idempotent state of the $C^*$-algebra $C^*_{\{p_{s,t}\}}(\cali{A})$ with respect to the unit $\{p_{s,t}\}_{0<s<t}$, in analogy with the concept of idempotent state of a compact quantum semigroup \cite{FS}.  We then show in Theorem \ref{DK1} and Proposition \ref{DK2} that any co-unit of a unital $C^*$-subproduct  system $\cali{A}$ gives rise to an idempotent state of the $C^*$-algebra $C^*_{\{ p_{s,t}\}}(\cali{A})$, with respect to a unit $\{p_{s,t}\}_{0<s<t}$, and that all co-units of $\cali{A}$ can be obtained in this way. We also show, in Proposition \ref {harici}, that the Hilbert spaces obtained by applying the GNS construction to the constituent states of a co-unit $\{\varphi_{s,t}\}_{0<s<t}$ of a $C^*$-subproduct system $\cali{A}$ form a Tsirelson subproduct system of Hilbert spaces $\cali{H}_{\{\varphi_{s,t}\}}$. We conclude this article by showing in Theorem \ref{harici1} that the Bhat-Mukherjee dilation of $\cali{H}_{\{\varphi_{s,t}\}}$ is isomorphic to the Tsirelson product system of Hilbert spaces $\cali{H}_{\{\varphi_{s,t}^\sharp\}}$ associated with the co-unit dilation $\{\varphi_{s,t}^\sharp\}_{0<s<t}$ of $\{\varphi_{s,t}\}_{0<s<t}$, given by Proposition \ref{DK}.

The main concepts and results of each section, as described above, will be illustrated concretely by using $C^*$-subproduct systems of commutative algebras as a model. Such systems are obtained from two-parameter multiplicative system of locally compact Hausdorff spaces by means of the Gelfand duality, as indicated in Example \ref{examp1}. The techniques used to investigate two-parameter multiplicative systems of compact Hausdorff spaces are based on the use of projective limits, mirroring the inductive limit techniques used in the study of $C^*$-subproduct systems.

We end this introductory section with a brief description of the notation used in this article, most of which is largely standard \cite{Arveson-book}\cite{Tak}. For general facts on inductive limits of $C^*$-algebras, we refer the reader to \cite{Sak}. 

In this paper, Hilbert spaces are not necessarily considered separable. The algebra of all bounded linear operators on a Hilbert space $H$ will be denoted $\cali{B}(H)$. Given two $C^*$-algebras $\m{A}$ and $\m{B}$, we will denote by $\m{A}\odot \m{B}$ their algebraic tensor product and by $\m{A}\otimes \m{B}$ their injective tensor product. For a state $\varphi$ of a $C^*$-algebra $\m{A}$, we shall denote by $\m{N}_\varphi=\{x\in\m{A}\,|\,\varphi(x^*x)=0\}$ its left kernel. The Borel $\sigma$-algebra on a topological space $X$ will be denoted $\op{Bor}(X)$. The pushforward measure of a measure $\mu$ on a measurable space $X$ along a measurable function $\chi$ on $X$ will be denoted $\chi_*\mu$.


\section{Background: definitions and examples}\label{3b100}
We begin this section by introducing some of the main concepts studied in this article. 
\begin{definition}\label{3b1}
A two-parameter co-multiplicative tensorial system of $C^*$-algebras, hereinafter referred to as a tensorial $C^*$-system, is a pair $\cali{A}=\left(\{\m{A}_{s,t}\}_{0<s<t}, \{\Delta_{r,s,t}\}_{0<r<s<t}\right)$
consisting of a family $\{\m{A}_{s,t}\}_{0<s<t}$ of $C^*$-algebras $\m{A}_{s,t}$, and a family $\{\Delta_{r,s,t}\}_{0<r<s<t}$ 
 of *-homomorphisms $\Delta_{r,s,t}: \m{A}_{r,t} \to \m{A}_{r,s}\otimes\m{A}_{s,t}$, called the co-multiplication of the tensorial $C^*$-system $\cali{A}$, which is co-associative in the sense that 
 the diagram
\begin{eqnarray}\label{Nov10d}
\xymatrixcolsep{5pc}\xymatrix@R+=1cm{\ar @{} [dr] 
\m{A}_{r,u}  \ar[d]^-{\Delta_{r,t,u}}\ar[r]^-{\Delta_{r,s,u}} & \m{A}_{r,s}\otimes  \m{A}_{s,u}\ar[d]^-{\op{id}_{\m{A}_{r,s}}\otimes \Delta_{s,t,u}} \\
 \m{A}_{r,t}\otimes  \m{A}_{t,u} \ar[r]_-{ \Delta_{r,s,t}\otimes \op{id}_{\m{A}_{t,u}} } & \m{A}_{r,s}\otimes  \m{A}_{s,t}\otimes  \m{A}_{t,u} }
\end{eqnarray}
commutes for all positive real numbers $0<r < s < t < u$, that is, $\left(\op{id}_{\m{A}_{r,s}}\otimes \Delta_{s,t,u}\right)\Delta_{r,s,u}=
\left( \Delta_{r,s,t}\otimes \op{id}_{\m{A}_{t,u}}  \right)\Delta_{r,t,u}$.
Here, as well as throughout this paper, $\op{id}_{\m{A}}$ denotes the identity mapping on a $C^*$-algebra $\m{A}$.
\end{definition}
The notion of unit of a tensorial $C^*$-system, defined below, is the operator algebraic counterpart of the notion of normalized unit of an Arveson product system of Hilbert spaces \cite[Section 4]{Arveson89}. The dual notion of co-unit will be defined and explored in Section \ref{sec5.1}.
\begin{definition} Let $\cali{A}=\left(\{\m{A}_{s,t}\}_{0<s<t}, \{\Delta_{r,s,t}\}_{0<r<s<t}\right)$ be a tensorial $C^*$-system. A unit of $\cali{A}$ is a family $\{p_{s,t}\}_{0<s<t}$ of non-zero projections $p_{s,t}\in \m{A}_{s,t}$ that satisfy the relation $\Delta_{r,s,t}(p_{r,t})=p_{r,s}\otimes p_{s,t}$, for all $0<r<s<t$. Any tensorial $C^*$-system that admits a units will be called unital. 
\end{definition}
One can naturally define the notion of morphism of tensorial $C^*$-systems, as follows.
\begin{definition} Suppose that $\cali{A}=(\{\m{A}_{s,t}\}_{0<s<t}, \{\Delta_{r,s,t}\}_{0<r<s<t})$ and $\cali{B}=\left(\{\m{B}_{s,t}\}_{0<s<t}, \{\Gamma_{r,s,t}\}_{0<r<s<t}\right)$ are tensorial  $C^*$-systems. A morphism from $\cali{A}$ to $\cali{B}$ is a family $\{\theta_{s,t}\}_{0<s<t}$ of *-homomorphisms $\theta_{s,t}:\m{A}_{s,t}\rightarrow \m{B}_{s,t}$ that make the diagram
\begin{eqnarray}\label{Nov10d}
\xymatrixcolsep{5pc}\xymatrix@R+=1cm{\ar @{} [dr] 
\m{A}_{r,t}  \ar[d]^-{\Delta_{r,s,t}}\ar[r]^-{\theta_{r,t}} & \m{B}_{r,t}\ar[d]^-{ \Gamma_{r,s,t}} \\
 \m{A}_{r,s}\otimes  \m{A}_{s,t} \ar[r]_-{ \theta_{r,s}\otimes\theta_{s,t} } & \m{B}_{r,s}\otimes  \m{B}_{s,t} }
\end{eqnarray}
 commute for all $0<r<s<t$, that is  $\Gamma_{r,s,t}\theta_{r,t}=(\theta_{r,s}\otimes \theta_{s,t}) \Delta_{r,s,t}.$ 
 
 A morphism $\{\theta_{s,t}\}_{0<s<t}$ is said to be a monomorphism, respectively an 
isomorphism, of tensorial $C^*$-systems if each *-homomorphism $\theta_{s,t}$ is a *-monomorphism, respectively *-isomorphism. \end{definition}
We focus mainly on the following two classes of tensorial $C^*$-systems.
\begin{definition}\label{May12022} A tensorial $C^*$-system $\cali{A}=(\{\m{A}_{s,t}\}_{0<s<t}, \{\Delta_{r,s,t}\}_{0<r<s<t})$ is said to be a $C^*$-subproduct system, respectively a $C^*$-product system, if the co-multiplication $\{\Delta_{r,s,t}\}_{0<r<s<t}$ consists of *-monomorphisms, respectively *-isomorphisms, of $C^*$-algebras.
\end{definition}

The notion of dilation of a $C^*$-subproduct system also plays a central role in our study. 
\begin{definition} A $C^*$-product system dilation of a $C^*$-subproduct system $\cali{A}$ consists of a $C^*$-product system $\cali{B}$ and a monomorphism from $\cali{A}$ to $\cali{B}$. \end{definition}
\begin{remarks}\label{Feb21-22}
(a) The class of tensorial $C^*$-systems is directly related to the class of one-parameter tensorial $C^*$-systems, i.e.,  pairs $\left(\{\m{Z}_{t}\}_{t>0}, \{\Xi_{s,t}\}_{0<s<t}\right)$ 
consisting of a family of $C^*$-algebras $\{\m{Z}_{t}\}_{t>0}$ and a family $\{\Xi_{s,t}\}_{0<s<t}$ 
 of *-homomorphisms $\Xi_{s,t}: \m{Z}_{s+t} \to \m{Z}_{s}\otimes\m{Z}_{t}$ that satisfy the co-associativity law $$\left(\op{id}_{\m{Z}_{r}}\otimes \Xi_{s,t}\right)\Xi_{r,s+t}=
\left( \Xi_{r,s}\otimes \op{id}_{\m{Z}_{t}}  \right)\Xi_{r+s,t},$$ for all $r,\,s,\, t>0$. The transition from one class to another can be achieved by adapting to our setting Tsirelson's arguments for systems of Hilbert spaces from \cite{Tsi03}. Although we do not intend to explore these connections in depth in this paper, we will briefly describe them below for future reference.

It is straightforward to obtain tensorial $C^*$-systems from one-parameter tensorial $C^*$-systems by simply setting $\m{A}_{s,t}=\m{Z}_{t-s}$, for all $0<s<t$, and $\Delta_{r,s,t}=\Xi_{s-r, t-s}$, for all $0<r<s<t$. The construction of one-parameter systems from two-parameter systems is a bit more involved, requiring a few changes and additions to the the original structure of a tensorial $C^*$-system. More precisely, the system $(\{\m{A}_{s,t}\}_{0\leq s<t}, \{\Delta_{r,s,t}\}_{0\leq r<s<t})$ should (i) be considered over the interval $[0,\infty)$, instead of $(0,\infty)$; (ii) be trivial, in sense that each $C^*$-algebra $\m{A}_{s,t}$ is *-isomorphic to a given $C^*$-algebra $\m{A}$, for all $0\leq s<t$; and (iii) be endowed with a homogeneous flow, i.e., a family $\{\alpha_{s,t}^h\}_{h\geq 0, 0\leq s<t}$ of *-isomorphisms of $C^*$-algebras $\alpha_{s,t}^h:\m{A}_{s,t}\to\m{A}_{s+h, t+h}$ that satisfy the following conditions: $\alpha_{s,t}^0=\op{id}_{\m{A}_{s,t}}$, $\alpha_{s+h, t+h}^{h'}\alpha_{s,t}^h=\alpha_{s,t}^{h+h'}$, for all $h,\, h'\geq 0$ and $0\leq s<t$, and $(\alpha_{r,s}^h\otimes\alpha^h_{s,t})\Delta_{r,s,t}= \Delta_{r+h,s+h,t+h}\alpha_{r,t}^h,$ for all  $h\geq 0$ and $0\leq r<s<t$. 

If these requirements are met, then by defining $\m{Z}_t=\m{A}_{0,t}$ for all $t>0$ and
$\Xi_{s,t}=\left(\op{id}_{\m{A}_{0,s}}\otimes (\alpha_{0,t}^s)^{-1}\right)\Delta_{0,s,s+t}$, for all $0<s<t$, we obtain that $\left(\{\m{Z}_{t}\}_{t>0}, \{\Xi_{s,t}\}_{0<s<t}\right)$ is a one-parameter $C^*$-tensorial system.

(b) Tensorial $W^*$-systems, including $W^*$-subproduct and product systems, can be defined in a similar way to tensorial $C^*$-systems by making some obvious adjustments. We note that continuous tensor product systems of $W^*$-algebras were also considered by V. Liebscher in \cite{Liebscher}, being defined as $W^*$-product systems over $[0,\infty)$, endowed with a homogeneous flow, as above. Consequently, they can be transformed into one-parameter $W^*$-product systems, following the procedure described in part (a).
\end{remarks}
In the rest of this section, we present a few relevant examples of tensorial $C^*$-systems, $C^*$-subproduct systems and $C^*$-product systems. We begin with two classes of examples that directly highlight the potential of tensorial $C^*$-systems to be studied in connection with both the theory of $C^*$-bialgebras \cite{Woro} and the theory of subproduct and product systems of Hilbert spaces. This perspective will be constantly maintained throughout this article.
\begin{example}\label{quantgr}
Let $(\m{A}, \Delta)$ be a $C^*$-bialgebra, consisting of a $C^*$-algebra $\m{A}$ and a co-multiplication
$\Delta:\m{A}\rightarrow \m{A}\otimes \m{A}$. Define $\m{A}_{s,t}=\m{A}$, for all $0<s<t$, and $\Delta_{r,s,t}=\Delta$, for all $0<r<s<t$. The resulting system
$\cali{A}_{\op{triv}}=\left(\{\m{A}_{s,t}\}_{0<s<t}, \{\Delta_{r,s,t}\}_{0<r<s<t}\right)$ is a $C^*$-tensorial system, which we call the trivial tensorial $C^*$-system associated with the $C^*$-bialgebra $(\m{A}, \Delta)$.
\end{example}
\begin{example}\label{exa1}
Let $\cali{H}=\left(\{H_{s,t}\}_{0<s<t}, \{U_{r,s,t}\}_{0<r<s<t}\right)$ be a Tsirelson subproduct system of finite dimensional Hilbert spaces. Take $\m{A}_{s,t}=\cali{B}(H_{s,t})$, for all $0<s<t$, and $\Delta_{r,s,t}=\op{ad}(U_{r,s,t})$, for all $0<r<s<t$, i.e., $\Delta_{r,s,t}(x)=U_{r,s,t}xU_{r,s,t,}^*$, for every $x\in \m{A}_{r,t}$. Then $\cali{A}_{\op{fin}}=\left(\{\m{A}_{s,t}\}_{0<s<t}, \{\Delta_{r,s,t}\}_{0<r<s<t}\right)$ is a $C^*$-subproduct system of matrix algebras. 

As in Remark \ref{Feb21-22}(a), one can use Arveson subproduct systems of finite dimensional Hilbert spaces to construct Tsirelson subproduct systems of finite dimensional Hilbert spaces, thus producing $C^*$-subproduct systems of matrix algebras from finite-dimensional Arveson subproduct systems. We note that a complete classification of Arveson subproduct systems of 2-dimensional Hilbert spaces was obtained by B. Tsirelson in \cite{Tsi09a, Tsi09b} (see also \cite{GS} for related results and applications).
\end{example}
Next, we focus on the description of tensorial $C^*$-systems of commutative $C^*$-algebras. This class will be examined in detail throughout this article.
\begin{example}\label{examp1} Let  $\cali{X}=\left(\{X_{s,t}\}_{0<s<t}, \{\chi_{r,s,t}\}_{0<r<s<t} \right)$ be a ``two-parameter multiplicative system of locally compact Hausdorff  spaces'', consisting of a family 
$\{X_{s,t}\}_{0<s<t}$ of locally compact Hausdorff spaces $X_{s,t}$, and a ``multiplication'' $\{\chi_{r,s,t}\}_{0<r<s<t}$, comprised of continuous functions $\chi_{r,s,t}:X_{r,s}\times X_{s,t}\rightarrow X_{r,t}$ that are proper and satisfy the associativity law \begin{eqnarray}\label{assiciative property1}\chi_{r,t,u}(\chi_{r,s,t}\times \op{id}_{X_{t,u}})=\chi_{r,s,u}(\op{id}_{X_{r,s}}\times \chi_{s,t,u}),\end{eqnarray} for all
 all $0<r<s<t<u$. Consider the commutative $C^*$-algebra  $\m{A}_{s,t}=C_0(X_{s,t})$ of all complex-valued continuous functions on $X_{s,t}$ vanishing at infinity, for all $0<s<t$, and the *-homomorphism $\Delta_{r,s,t}:C_0(X_{r,t})\rightarrow C_0(X_{r,s})\otimes C_0(X_{s,t})$ induced by $\chi_{r,s,t}$, i.e., $\Delta_{r,s,t}f= f\comp\chi_{r,s,t},$
for all $f\in C_0(X_{r,t})$ and $0<r<s<t$, which is obtained by identifying  the $C^*$-algebras 
$C_0(X_{r,s})\otimes C_0(X_{s,t})$ and $C_0(X_{r,s}\times X_{s,t})$. The resulting system $\cali{A}_{\op{com}}=\left(\{\m{A}_{s,t}\}_{0<s<t}, \{\Delta_{r,s,t}\}_{0<r<s<t}\right)$
is a tensorial $C^*$-system. This system is a $C^*$-subproduct system, respectively a $C^*$-product system, provided that the functions $\chi_{r,s,t}$ are all surjective, respectively bijective. Conversely, using the Gelfand representation and the Gelfand-Kolmogorov theorem, we can easily deduce that any tensorial $C^*$-system $\left(\{\m{A}_{s,t}\}_{0<s<t}, \{\Delta_{r,s,t}\}_{0<r<s<t}\right)$ of commutative $C^*$-algebras $\m{A}_{s,t}$ with proper *-homomorphism $\Delta_{r,s,t}$ arises from a two-parameter multiplicative system of locally compact Hausdorff  spaces.

It is clear that the units of $\cali{A}_{\op{com}}$ are those families $\{1_{Y_{s,t}}\}_{0<s<t}$ of indicator functions $1_{Y_{s,t}}$ of open compact sets $Y_{s,t}\subset X_{s,t}$ satisfying the compatibility condition $1_{Y_{r,t}}\comp\chi_{r,s,t}=1_{Y_{r,s}\times Y_{s,t}}$, for all $0<r<s<t$. In particular, if the spaces $X_{s,t}$ are all connected, then  $\cali{A}_{\op{com}}$ does not have non-trivial units.

We illustrate the above concept through a direct construction based on a single compact Hausdorff space $X$. Let $X_{s,t}=\prod_{(s,t]}X$ be the cartesian product of copies of X over the interval half-open $(s, t]$, for all $0 < s < t$, endowed with the product topology, and $\chi_{r,s,t}:X_{r,s}\times X_{s,t}\rightarrow X_{r,t}$ be the gluing function \begin{eqnarray}\label{cucu7}\chi_{r,s,t}(f,g)(x)=\left\{\begin{array}{cc}f(x),&\mbox{if}\;r<x\leq s\\ g(x),&\mbox{if}\;s<x\leq t \end{array}\right.\end{eqnarray} for all $f\in {X}_{r,s}$,  $g\in {X}_{s,t}$, and $r<x\leq t$. Then the resulting system $\left(\{X_{s,t}\}_{0<s<t}, \{\chi_{r,s,t}\}_{0<r<s<t} \right)$ is a two-parameter multiplicative system of compact Hausdorff spaces, and its associated tensorial $C^*$-system is a $C^*$-product system. 
\end{example}
By adapting the above method to the setting of locally compact groups, we can also obtain tensorial $C^*$-systems of reduced group $C^*$-algebras, as follows.

\begin{example}\label{cucuu}
Let $\cali{G}=
\left(\{G_{s,t}\}_{0<s<t}, \{\mu_{s,t}\}_{0<s<t},\{\chi_{r,s,t}\}_{0<r<s<t} \right)$ be a ``two-parameter multiplicative system of locally compact groups'', consisting of a family 
$\{G_{s,t}\}_{0<s<t}$ of locally compact groups $G_{s,t}$ with left Haar measure $\mu_{s,t}$ and Haar modulus $\mathit{\Delta}_{G_{s,t}}$, and a family $\{\chi_{r,s,t}\}_{0<r<s<t}$ of measurable and measure-preserving group homomorphisms $\chi_{r,s,t}:G_{r,s}\times G_{s,t}\rightarrow G_{r,t}$ that are modular invariant, in the sense that that $\mathit{\Delta}_{G_{r,s}\times G_{s,t}} =\mathit{\Delta}_{G_{r,t}}\comp \chi_{r,s,t}$, for all $0<r<s<t<u$, and that satisfy the associativity law (\ref{assiciative property1}). 
For any two real numbers $0<s<t$, consider the reduced group $C^*$-algebra $\m{A}_{s,t}=C^*_r(G_{s,t})$ of the group $G_{s,t}$. By identifying the injective tensor product of Banach spaces $L^1(G_{r,s})\otimes L^1(G_{s,t})$ with the Banach space $L^1(G_{r,s}\times G_{s,t})$, for all $0<r<s<t$, we can consider the operators $\Delta_{r,s,t}:L^1(G_{r,t})\rightarrow L^1(G_{r,s})\otimes L^1(G_{s,t})$ induced by the homomorphisms $\chi_{r,s,t}$, that is $\Delta_{r,s,t}f=f\comp \chi_{r,s,t},$ for all $f\in L^1(G_{r,t})$. Each operator $\Delta_{r,s,t}$ admits a unique extension to a *-homomorphism, denoted by the same letter, $\Delta_{r,s,t}: \m{A}_{r,t} \to \m{A}_{r,s}\otimes\m{A}_{s,t}$, and 
the resulting system $\cali{A}_{\op{red}}=\left(\{\m{A}_{s,t}\}_{0<s<t}, \{\Delta_{r,s,t}\}_{0<r<s<t}\right)$ is a tensorial $C^*$-system.

Proceeding as in Example \ref{examp1}, one can immediately construct a two-parameter multiplicative system of compact groups starting with a  compact group $G$ with (normalized) Haar measure $\mu$:  set $G_{s,t}=\prod_{(s,t]}G$, for all $0<s<t$, and define $\chi_{r,s,t}:G_{r,s}\times G_{s,t}\rightarrow G_{r,t}$ as in (\ref{cucu7}), for all $0<r<s<t$.
\end{example}

In the following example, we construct $C^*$-product systems from measurable Arveson systems of Hilbert spaces. We only outline the main steps of this construction, whose full details and consequences will be discussed elsewhere.

\begin{example}\label{remus}
Let  $E=\left(\{E_{t}\}_{t>0}, \{V_{s,t}\}_{s,t>0}\right)$ be a measurable Arveson product system of Hilbert spaces. Consider the direct integral Hilbert space $L^2(E)=\int _{(0,\infty)}^\oplus E_t\,dt$ and the left regular representation $\ell$ of $E$ on  
 $L^2(E)$, given by $$(\ell_xf)(s)=\left\{\begin{array}{cc} V_{t,s-t}^{-1}(x\otimes  f(s-t)),&\mbox{if}\;s>t\\0,&\mbox{if}\;0<s\leq t
 \end{array} \right.$$ for all $f\in L^2(E)$,  $x\in E_t$ and $t>0$ (see \cite[Prop. 3.3.1]{Arveson-book}). We also consider the $E$-semigroup $\alpha=\{\alpha_t\}_{t\geq 0}$ of $\cali{B}(L^2(E))$ induced by $\ell$, i.e., $$\alpha_t(x)=\sum_{n=1}^\infty \ell_{e_n(t)}x \ell_{e_n(t)}^*,$$ for all $x\in\cali{B}(L^2(E))$, where
 $\{e_n(t)\}_{n\geq 1}$ is an orthonormal basis for $E_t$ chosen so that the mapping  $t\mapsto e_n(t)$ is Borel measurable, for all positive integers $n$.
 For every $t>0$, let also $A_t$ be the norm-closed linear span $$A_t= \overline{\op{span}}\{\ell_x\ell_y^*\,|\, x,\, y\in E_{t}\}.$$ The spaces $A_t$ are $C^*$-algebras that are isomorphic to the $C^*$-algebra of compact operators, and satisfy the relation
 $A_sA_t\subseteq A_{\max (s,t)}$ for all $s,\,t>0$.
Using this, we define the $C^*$-algebra $$
\m{A}_{s,t} = \alpha_s(A_{t-s}),$$ for all  $0<s<t$.
Noticing that the $C^*$-algebras $
\m{A}_{r,s}$ and $
\m{A}_{s,t}$ commute, and $
\m{A}_{r,s}
\m{A}_{s,t}\subseteq
\m{A}_{r,t}$,  for all $0<r<s<t$, we deduce that the multiplication mapping
$$\m{A}_{r,s}\odot \m{A}_{s,t}\ni x\otimes y\mapsto xy\in \m{A}_{r,t}$$ induces a unique isomorphism of $C^*$-algebras $\Delta_{r,s,t}:\m{A}_{r,s}\otimes \m{A}_{s,t}\to\m{A}_{r,t}$, for all $0<r<s<t$. The resulting system $\cali{A}_{\op{ps}}=\left(\{\m{A}_{s,t}\}_{0<s<t}, \{\Delta_{r,s,t}^{-1}\}_{0<r<s<t}\right)$ is a $C^*$-product system. It is clear that if $\{u(t)\}_{t>0}$ is a normalized unit of the Arveson product system $E$ and $p_{s,t}=\alpha_s(\ell_{u(t-s)}\ell_{u(t-s)}^*)$, for all $0<s<t$, then $\{p_{s,t}\}_{0<s<t}$ is a unit of the $C^*$-product system $\cali{A}_{\op{ps}}$.
\end{example}

 
 \section{The inductive dilation of a $C^*$-subproduct system}\label{ch3.1}
 Let $\cali{A}=\left(\{\m{A}_{s,t}\}_{0<s<t}, \{\Delta_{r,s,t}\}_{0<r<s<t}\right)$ be a $C^*$-subproduct system. For any two positive real numbers  $0<s < t$,  consider the partially ordered set $(\cali{P}_{s,t},\subseteq)$ of all finite partitions of the interval $[s,t]$, ordered by inclusion. For any partition $I\in \cali{P}_{s,t} $ of the form $I=\{s=\iota_0<\iota_1<\iota_2<\,\dots<\iota_m<\iota_{m+1}=t\},$ we define the $C^*$-algebra \begin{eqnarray}\label{jan15}\m{A}_I=
\m{A}_{\iota_0, \iota_1}\otimes \m{A}_{\iota_1,\iota_2}\otimes \dots\otimes \m{A}_{\iota_m,\iota_{m+1}}.\end{eqnarray}
 If $J\in \cali{P}_{s,t}$ is a refinement of $I$, i.e., $I\subseteq J$, then the partition $J$ can decomposed, with respect to $I$, as \begin{eqnarray}\label{decoopa}J=I_0\cup I_1\cup \cdots \cup I_m,\end{eqnarray} where $I_i=\{j\in J, \iota_i\leq j\leq \iota_{i+1}\}=\{\iota_i=\iota_{i_0}<\iota_{i_1}<\dots<\iota_{i_{n_{I_i}}}<\iota_{i+1}\}\in \cali{P}_{\iota_i,\iota_{i+1}},$ for some $n_{I_i}\in\mathbb{N}$. 
 Accordingly, the $C^*$-algebra $\m{A}_J$ can be decomposed as $
\m{A}_J=\m{A}_{I_0}\otimes \m{A}_{I_1}\otimes\dots\otimes \m{A}_{I_m}.$

\begin{definition}\label{Jan21}Let $I=\{s=\iota_0<\iota_1<\iota_2<\,\dots<\iota_m<\iota_{m+1}=t\}\in \cali{P}_{s,t}$.\\ (a) We consider the *-monomorphism $\Delta_{\{s,t\},I}:\m{A}_{s,t}\to \m{A}_{I}$, defined iteratively as follows: \begin{eqnarray*}\Delta_{\{s,t\},I}=\left\{\begin{array}{llll}\Delta_{\iota_0,\iota_1,\iota_2}, &\mbox{if}\;m=1\\\left(\Delta_{\{\iota_0,\iota_m\},I\setminus\{\iota_{m+1}\}}\otimes \op{id}_{\iota_m,\iota_{m+1}}\right)\Delta_{\iota_0,\iota_m,\iota_{m+1}},
& \mbox{if}\; m\geq 2
\end{array}\right.\end{eqnarray*}
Here $I\setminus\{\iota_{m+1}\}\in \cali{P}_{s,\iota_m}$ is the partition obtained by removing the endpoint $t=\iota_{m+1}$ from $I$, and $\op{id}_{\iota_m,\iota_{m+1}}$
is the identity mapping on $\m{A}_{\iota_m,\iota_{m+1}}$.\\
(b) For any refinement $J\in \cali{P}_{s,t}$ of $I$, decomposed as $J=I_0\cup \ldots \cup I_m$ with respect to $I$, we consider the *-monomorphism $\Delta_{I,J}:\m{A}_I\rightarrow \m{A}_J$, \begin{eqnarray*}\label{jan16}
\Delta_{I,J}=\Delta_{\{\iota_0,\iota_1\}, I_0}\otimes \Delta_{\{\iota_1,\iota_2\}, I_1}\otimes \cdots \Delta_{\{\iota_m,\iota_{m+1}\}, I_m}.\end{eqnarray*}
For $I=J$, we set $\Delta_{I,I}=\op{id}_{\m{A}_I}$.
\end{definition}
\begin{remark}\label{Nov14a}The *-monomorphism $\Delta_{\{s,t\},I}$ can be expanded as
$\Delta_{\{s,t\},I}=(\Delta_{\iota_0,\iota_1,\iota_2}\otimes \op{id}_{\iota_2,\iota_3}\otimes\cdots \op{id}_{\iota_m,\iota_{m+1}})
(\Delta_{\iota_0,\iota_2,\iota_3}\otimes \op{id}_{\iota_3,\iota_4}\otimes\dots \op{id}_{\iota_m,\iota_{m+1}})
\cdots\Delta_{\iota_0,\iota_m,\iota_{m+1}}\\
= (\op{id}_{\iota_0,\iota_1}\otimes \dots \otimes  \op{id}_{\iota_{m-3},\iota_{m-2}}\otimes
\Delta_{\iota_{m-1},\iota_m,\iota_{m+1}})\cdots ( \op{id}_{\iota_0,\iota_1}\otimes \Delta_{\iota_1,\iota_2,\iota_{m+1}})\Delta_{\iota_0,\iota_1,\iota_{m+1}},$
for all $m\geq 2$.
\end{remark}
\begin{lemma}\label{Aug17} Let $I,\, J\in \cali{P}_{s,t}$, $I\subseteq J$. We then have
\begin{enumerate}\item[(i)] $\Delta_{\{s,t\},J}=\Delta_{I,J}\Delta_{\{s,t\}, I};$
\item[(ii)] $\Delta_{I,J}=\Delta_{I\cap[s,u],J\cap[s,u]}\otimes\Delta_{I\cap[u,t], J\cap[u,t]},$ for every $u\in I$, $s<u<t$.
\end{enumerate}
\end{lemma}
\begin{proof} (i) Without loss of generality, one can assume that the partition $I$ is of the form $I=\{s=\iota_0<\iota_1<\iota_{2}=t\}$. The refinement $J$ can then be written as $J=I_0\cup I_1$,
where $I_i=\{\iota_i=\iota_{i_0}<\iota_{i_1}<\dots<\iota_{i_{n_{I_i}}}<\iota_{i+1}\}\in \cali{P}_{\iota_i,\iota_{i+1}}$, for some $n_{I_i}\in\mathbb{N}$, $i\in\{0,1\}$. Using Remark \ref{Nov14a} and the co-associativity law (\ref{Nov10d}), we deduce that
\begin{eqnarray*}\label{crazy1}
(\op{id}_{s,\iota_1}\otimes \Delta_{\{\iota_1,t\}, I_1})\Delta_{s,\iota_1,t}=(\Delta_{\iota_0,\iota_1,\iota_{1_1}}\otimes \op{id})\cdots (\Delta_{\iota_0,\iota_1,\iota_{1_{n_{I_1}-1}}}\otimes \op{id})\Delta_{s,  \iota_{1_{n_{I_1}}}, \iota_2}.
\end{eqnarray*} Consequently, we obtain that
\begin{eqnarray*}
&&\Delta_{I,J}\Delta_{\{s,t\}, I}=(\Delta_{\{s,\iota_1\}, I_0}\otimes \Delta_{\{\iota_1,t\}, I_1})\Delta_{s,\iota_1,t}\\&=&(\Delta_{\{s,\iota_1\}, I_0}\otimes \op{id}_{I_1})(\op{id}_{s,\iota_1}\otimes \Delta_{\{\iota_1,t\}, I_1})\Delta_{s,\iota_1,t}\\
&=&
[(\Delta_{\iota_0,\iota_{0_1},\iota_{0_2}}\otimes \op{id}  )(\Delta_{\iota_0,\iota_{0_2},\iota_{0_3}}\otimes \op{id}  )\cdots 
(\Delta_{\iota_0,\iota_{0_{n_{I_0}-1}},\iota_{0_{n_{I_0}}}}\otimes \op{id})\Delta_{\iota_0,\iota_{0_{n_{I_0}}},\iota_{1}}\otimes \op{id}]\\
&\comp&(\op{id}_{s,\iota_1}\otimes \Delta_{\{\iota_1,t\}, I_1})\Delta_{s,\iota_1,t}\\&=&
(\Delta_{\iota_0,\iota_{0_1},\iota_{0_2}}\otimes \op{id}  )(\Delta_{\iota_0,\iota_{0_2},\iota_{0_3}}\otimes \op{id}  )\cdots 
(\Delta_{\iota_0,\iota_{0_{n_{I_0}-1}},\iota_{0_{n_{I_0}}}}\otimes \op{id})(\Delta_{\iota_0,\iota_{0_{n_{I_0}}},\iota_{1}}\otimes \op{id})\\
&\comp&(\Delta_{\iota_0,\iota_1,\iota_{1_1}}\otimes \op{id})\cdots (\Delta_{\iota_0,\iota_1,\iota_{1_{n_{I_1}-1}}}\otimes \op{id})\Delta_{s,  \iota_{1_{n_{I_1}}}, \iota_2}\\&=&\Delta_{\{s,t\}, J}.
\end{eqnarray*} 
(ii) Suppose that $I=\{s=\iota_0<\iota_1<\iota_2<\,\dots<\iota_m<\iota_{m+1}=t\}\in \cali{P}_{s,t}$, and let $ J=I_0\cup \ldots \cup I_m$ be the decomposition of $J$ with respect to $I$. Let $k$ be a positive integer, $1\leq k\leq m$, such that $u=\iota_k$. Then $I\cap[s,u]\in \cali{P}_{s,\iota_k}$, $I\cap[s,u]\subseteq J\cap[s,u]$, $I\cap[u,t]\in \cali{P}_{\iota_k, t}$, $I\cap[u,t]\subseteq J\cap[u,t]$, $J\cap[s,u]=I_0\cup \ldots \cup I_{k-1}$, and  $J\cap[u,t]=I_k\cup \ldots \cup I_m$. It follows that
\begin{eqnarray*}&&\Delta_{I\cap[s,u],J\cap[s,u]}\otimes\Delta_{I\cap[u,t], J\cap[u,t]}=\Delta_{\{\iota_0,\iota_1\}, I_0}\otimes \Delta_{\{\iota_1,\iota_2\}, I_1}\otimes \cdots \Delta_{\{\iota_{k-1},\iota_{k}\}, I_k}\\
&\otimes&\Delta_{\{\iota_k,\iota_{k+1}\}, I_k}\otimes \Delta_{\{\iota_{k+1},\iota_{k+2}\}, I_{k+1}}\otimes \cdots \Delta_{\{\iota_m,\iota_{m+1}\}, I_m}=\Delta_{I,J},
\end{eqnarray*}
as claimed.
\end{proof}

\begin{proposition}\label{lemma inductive limit-c} Let $\cali{A}=(\{\m{A}_{s,t}\}_{0<s<t}, \{\Delta_{r,s,t}\}_{0<r<s<t})$ be a $C^*$-subproduct system. For any two positive real numbers  $0<s < t$, the system
$$
\Bigl\{(\m{A}_I,\Delta_{I,J})\,|\, I,\,J \in \cali{P}_{s,t},\; I\subseteq J \Bigr\}
$$
is an inductive system of $C^*$-algebras over the partially ordered set $( \cali{P}_{s,t},\subseteq )$.
\end{proposition}

\begin{proof} Consider three arbitrary partitions $I,\,J,\,K\in \cali{P}_{s,t}$ such that $I\subseteq J\subseteq  K$. We claim that the compatibility condition $$\Delta_{I,K}=\Delta_{J,K} \Delta_{I,J}$$
is satisfied. For this purpose, suppose that $I=\{s=\iota_{0}<\iota_{1}<\iota_{2}<\,\dots<\iota_{m}<\iota_{{m+1}}=t\}$, and let $J=I_0\cup \cdots \cup I_m$ be the decomposition of $J$ with respect to $I$, where for each $0\leq i\leq m$, $I_i=\{\iota_i=\iota_{i_0}<\iota_{i_1}<\dots<\iota_{i_{n_{I_i}}}<\iota_{i+1}\}\in \cali{P}_{\iota_i,\iota_{i+1}}$, for some $n_{I_i}\in\mathbb{N}$. Let also $K=J_0\cup J_1\cup\dots\cup J_\ell$ be the decomposition of $K$ with respect to $J$, where $J_0\in  \cali{P}_{\iota_{0_0},\iota_{0_{1}}}$, $J_1\in  \cali{P}_{\iota_{0_1},\iota_{0_{2}}}$, $\cdots$, $J_{n_{I_0}}\in  \cali{P}_{\iota_{0_{n_{I_0}}},\iota_{{1}}}$, $\cdots$, $J_{\ell}\in  \cali{P}_{\iota_{m_{n_{I_m}}},t}$.
Putting it all together, we obtain that $$K=I_0'\cup I_1'\cup\dots\cup I_m',$$ where $I_0'=J_0\cup J_1\cup\dots \cup J_{n_{I_0}}$, $I_1'=J_{n_{I_0}+1}\cup\dots \cup J_{n_{I_1}}$, $\dots$, $I_m'=J_{n_{I_{m-1}}+1}\cup\dots \cup J_\ell$. This is exactly the partition decomposition of $K$ with respect to $I$.
 Since \begin{eqnarray*}
\Delta_{I_0, I_0'}&=&\Delta _{\{\iota_{0_0},\iota_{0_1}\}, J_0}\otimes \Delta _{\{\iota_{0_1},\iota_{0_2}\}, J_1}\otimes \dots\otimes \Delta _{\{\iota_{0_{n_{I_0}}},\iota_{1}\}, J_{n_{I_0}}}\\
\Delta_{I_1, I_1'}&=&\Delta _{\{\iota_{1_0},\iota_{1_1}\}, J_{n_{I_0}+1}}\otimes \Delta _{\{\iota_{1_1},\iota_{1_2}\}, J_{n_{I_0}+2}}\otimes \dots\otimes \Delta _{\{\iota_{1_{n_{I_1}}},\iota_{2}\}, J_{n_{I_1}}}\\
\cdots&\cdots&\cdots\cdots\cdots\\
\Delta_{I_m, I_m'}&=&\Delta _{\{\iota_{m_0},\iota_{m_1}\}, J_{n_{I_{m-1}}+1}}\otimes  \dots\otimes \Delta _{\{\iota_{m_{n_{I_m}}},\iota_{m+1}\}, J_\ell},
\end{eqnarray*}  
 we have that  $\Delta_{I_0, I_0'}\otimes \Delta_{I_1, I_1'} \otimes \cdots \otimes\Delta_{I_m, I_m'}=\Delta _{\{\iota_{0_0},\iota_{0_1}\}, J_0}\otimes \Delta _{\{\iota_{0_1},\iota_{0_2}\}, J_1}\otimes\dots \otimes\Delta _{\{\iota_{m_{n_{I_m}}},\iota_{m+1}\}, J_\ell}=\Delta_{J,K}.$
Using this identity and Lemma \ref{Aug17}, we obtain that\begin{eqnarray*}
&&\Delta_{I,K}=\Delta_{\{\iota_0,\iota_1\}, I_0'}\otimes \Delta_{\{\iota_1,\iota_2\}, I_1'}\otimes \cdots\otimes \Delta_{\{\iota_m,\iota_{m+1}\}, I_m'}\\
&=&\Delta_{I_0, I_0'}\Delta_{\{\iota_0,\iota_1\}, I_0}\otimes \Delta_{I_1, I_1'} \Delta_{\{\iota_1,\iota_2\}, I_1}\otimes \cdots \otimes\Delta_{I_m, I_m'}\Delta_{\{\iota_m,\iota_{m+1}\}, I_m}\\
&=&(\Delta_{I_0, I_0'}\otimes \Delta_{I_1, I_1'} \otimes \cdots \Delta_{I_m, I_m'})(\Delta_{\{\iota_0,\iota_1\}, I_0}\otimes \Delta_{\{\iota_1,\iota_2\}, I_1}\otimes \cdots \Delta_{\{\iota_m,\iota_{m+1}\}, I_m})\\
&=&\Delta_{J,K}(\Delta_{\{\iota_0,\iota_1\}, I_0}\otimes \Delta_{\{\iota_1,\iota_2\}, I_1}\otimes \cdots \otimes\Delta_{\{\iota_m,\iota_{m+1}\}, I_m})\\&=&\Delta_{J,K} \Delta_{I,J},
\end{eqnarray*}
as claimed. Therefore $\{(\m{A}_I,\Delta_{I,J})\,|\, I\subseteq J \in \cali{P}_{s,t}\}$ is an inductive system of $C^*$-algebras.
\end{proof}

\begin{definition}\label{Oct5} Let $\cali{A}=(\{\m{A}_{s,t}\}_{0<s<t}, \{\Delta_{r,s,t}\}_{0<r<s<t})$ be a $C^*$-subproduct system. For any two positive real numbers  $0<s < t$, we define the $C^*$-algebra \begin{eqnarray}\m{A}_{s,t}^\sharp=\limind\, \Bigl\{(\m{A}_I,\Delta_{I,J})\,|\, I,\,J \in \cali{P}_{s,t},\; I\subseteq J\Bigr\},\end{eqnarray}
as the $C^*$-inductive limit of the system $\{(\m{A}_I,\Delta_{I,J})\,|\, I,\,J \in \cali{P}_{s,t},\; I\subseteq J\}$. \end{definition}
We also consider the family $\{\Delta _I^\sharp\}_{I\in  \cali{P}_{s,t}}$ of connecting mappings  $\Delta _I^\sharp:\m{A}_I\to \m{A}_{s,t}^\sharp$, $I\in  \cali{P}_{s,t}$, associated with this inductive limit construction. Therefore the *-monomorphisms  $\Delta _I^\sharp$ satisfy the compatibility condition $$\Delta_{J}^\sharp\Delta_{I,J}=\Delta_I^\sharp,$$for all $I,\,J\in \cali{P}_{s,t}$, $I\subseteq J,$ and the norm-density condition $$\m{A}_{s,t}^\sharp=\overline{\bigcup_{I\in  \cali{P}_{s,t}}\Delta_I^\sharp(\m{A}_I)}^{\|\cdot\|}.$$ To ease the notation, if $I=\{s,t\}$ is the trivial partition, then we shall simply use the notation $\Delta _{s,t}^\sharp$ instead of $\Delta _{\{s,t\}}^\sharp$, whenever necessary.

\begin{theorem}\label{star-isomorphism theorem}Let $\cali{A}=\left(\{\m{A}_{s,t}\}_{0<s<t}, \{\Delta_{r,s,t}\}_{0<r<s<t}\right)$ be a $C^*$-subproduct system. 
For any real numbers $0<r<s<t$, there exists a $*$-isomorphism of $C^*$-algebras $\Delta_{r,s,t}^\sharp:\m{A}_{r,t}^\sharp\to \m{A}_{r,s}^\sharp\otimes \m{A}_{s,t}^\sharp$ so that 
the resulting system $\cali{A}^\sharp=\left(\{\m{A}_{s,t}^\sharp\}_{0<s<t},\,\{\Delta_{r,s,t}^\sharp\}_{0<r<s<t}\right)$ is a $C^*$-product system dilation of $\cali{A}$. 
\end{theorem}

\begin{proof}
Let $0<r<t$ be two given positive real numbers. For any real number $s$ with $0<r<s<t$, consider the family of partitions $$\cali{P}_{r,s}\lor \cali{P}_{s,t}=\{ I\cup J\,|\, I\in\cali{P}_{r,s}, J\in \cali{P}_{s,t}\}$$ of the interval $[r,t]$. We note that $\cali{P}_{r,s}\lor \cali{P}_{s,t}$ is a cofinal subset of the partially ordered set $(\cali{P}_{r,t},\subseteq)$. Consequently, the $C^*$-algebra $\m{A}_{r,t}^\sharp$ can be realized as the $C^*$-inductive limit of the inductive system $$\{(\m{A}_{I\cup J},\Delta_{I\cup J,I'\cup J'})\,|\, I\cup J,\,I'\cup J' \in \cali{P}_{r,s}\lor \cali{P}_{s,t},\; I\cup J\subseteq I'\cup J'\}.$$ Furthermore, we note that the mapping $\cali{P}_{r,s}\lor \cali{P}_{s,t}\ni I\cup J\mapsto (I, J)\in \cali{P}_{r,s}\times \cali{P}_{s,t}$ is an order isomorphism, where  $\cali{P}_{r,s}\times \cali{P}_{s,t}$ is endowed with the product order. Because $\m{A}_{I\cup J}=\m{A}_I\otimes\m{A}_J$ and $\Delta_{I\cup J,I^{\prime}\cup J^{\prime}}=\Delta_{I,I^{\prime}}\otimes \Delta_{J,J^{\prime}}$, for all $I,\, I^{\prime}\in\cali{P}_{r,s}$, $I\subseteq I^{\prime}$, and $J,\,J^{\prime}\in \cali{P}_{s,t}$, $J\subseteq J^{\prime}$, we infer that  the $C^*$-algebra $\m{A}_{r,t}^\sharp$ can also be realized as the $C^*$-inductive limit of the inductive system  $$\{(\m{A}_{I}\otimes \m{A}_{J} ,\Delta_{I,I^{\prime}}\otimes \Delta_{J, J^{\prime}} )\,|\, (I,J),\,(I^{\prime},J^{\prime}) \in \cali{P}_{r,s}\times\cali{P}_{s,t},\; I\subseteq I^{\prime},\;J\subseteq J^{\prime}\}.$$ Noticing that the families
 $\{\Delta_I^\sharp\otimes \Delta_J^\sharp\}_{ \cali{P}_{r,s}\times \cali{P}_{s,t}}$ and $\{\Delta_{I,I^{\prime}}\otimes \Delta_{J, J^{\prime}}\}_{ \cali{P}_{r,s}\times \cali{P}_{s,t}}$  are compatible, and that 
  the union of all range algebras $\Delta_I^\sharp\otimes \Delta_J^\sharp\left(\m{A}_I\otimes \m{A}_J\right)$ is norm-dense in $\m{A}_{r,s}^\sharp\otimes \m{A}_{s,t}^\sharp$, it follows from the universal property of the inductive limit that  there is a unique *-isomorphism $\Delta_{r,s,t}^\sharp:\m{A}_{r,t}^\sharp\rightarrow \m{A}_{r,s}^\sharp\otimes \m{A}_{s,t}^\sharp$ that satisfies the compatibility condition \begin{eqnarray}\label{compas}\Delta_{r,s,t}^\sharp\Delta_{I\cup J}^\sharp=\Delta_I^\sharp\otimes \Delta_J^\sharp,\end{eqnarray}
for all $I\in\cali{P}_{r,s}$, $J\in  \cali{P}_{s,t}$.

We claim that $\cali{A}^\sharp=\left(\{\m{A}_{s,t}^\sharp\}_{0<s<t},\,\{\Delta_{r,s,t}^\sharp\}_{0<r<s<t}\right)$ is a $C^*$-product system, i.e., that the family of *-monomorphisms $\{\Delta_{r,s,t}^\sharp\}_{0<r<s<t}$ satisfies the co-associativity law (\ref{Nov10d}). For this purpose, consider three arbitrary partitions  $I\in  \cali{P}_{r,s}$, $J\in \cali{P}_{s.t}$ and $K\in \cali{P}_{t,u}.$ Using  (\ref{compas}), we have
\begin{eqnarray*}
(\Delta_{r,s,t}^\sharp\otimes \op{id}_{\m{A}_{t,u}^\sharp})\Delta_{r,t,u}^\sharp \Delta_{I\cup J\cup K}^\sharp&=&(\Delta_{r,s,t}^\sharp\otimes \op{id}_{\m{A}_{t,u}^\sharp})(\Delta_{I\cup J}^\sharp\otimes \Delta_K^\sharp)\\&=&\Delta_I^\sharp\otimes \Delta_J^\sharp\otimes \Delta_K^\sharp
\end{eqnarray*}
 and similarly $(\op{id}_{\m{A}_{r,s}^\sharp}\otimes \Delta_{s,t,u}^\sharp)\Delta_{r,s,u}^\sharp \Delta_{I\cup J\cup K}^\sharp=\Delta_I^\sharp\otimes \Delta_J^\sharp\otimes \Delta_K^\sharp$. Because the set $\bigcup_{I,\,J,\,K}\Delta_{I\cup J\cup K}^\sharp (\m{A}_I\otimes \m{A}_J\otimes \m{A}_K)$ is everywhere dense in $\m{A}^\sharp_{t,u}$, the conclusion follows. 
 
 Finally, we show that the family $\{\Delta_{s,t}^\sharp\}_{0<s<t}$ is a monomorphism of tensorial $C^*$-systems, from $\cali{A}$ to $\cali{A}^\sharp$. For this, let $0<r<s<t$ be fixed real numbers. Applying (\ref{compas}) to the particular case $I=\{r,s\}\in \cali{P}_{r,s}$, $J=\{s,t\}\in \cali{P}_{s,t}$, we obtain that
$$\left(\Delta_{r,s}^\sharp\otimes \Delta_{s,t}^\sharp\right)\Delta_{r,s,t}=\Delta^\sharp_{r,s,t}\Delta^\sharp_{I\cup J}\Delta_{r,s,t}=\Delta^\sharp_{r,s,t}\Delta_{r,t}^\sharp,$$
as needed. The theorem is proved.
\end{proof}

\begin{definition}
 Let $\cali{A}=\left(\{\m{A}_{s,t}\}_{0<s<t}, \{\Delta_{r,s,t}\}_{0<r<s<t}\right)$ be a $C^*$-subproduct system. The $C^*$-product system dilation $$\cali{A}^\sharp=\left(\{\m{A}_{s,t}^\sharp\}_{0<s<t},\,\{\Delta_{r,s,t}^\sharp\}_{0<r<s<t}\right)$$ will be called the inductive dilation of $\cali{A}$. The monomorphism $\{\Delta_{s,t}^\sharp\}_{0<s<t}$ will be called the inductive embedding of $\cali{A}$ into $\cali{A}^\sharp$.
\end{definition}
\begin{observation}\label{hasmas}
We notice that the inductive dilation of a unital $C^*$-subproduct system $\cali{A}=\left(\{\m{A}_{s,t}\}_{0<s<t}, \{\Delta_{r,s,t}\}_{0<r<s<t}\right)$ is also unital. Indeed, suppose that $\{p_{s,t}\}_{0<s<t}$ is a unit of $\cali{A}$, and let $p_I=p_{\iota_0,\iota_1}\otimes p_{\iota_1, \iota_2}\otimes \cdots p_{\iota_m,\iota_{m+1}},$ for every partition $I\in \cali{P}_{s,t} $, $I=\{s=\iota_0<\iota_1<\iota_2<\,\dots<\iota_m<\iota_{m+1}=t\}$. Then $\Delta_{I, J}(p_I)=p_J$, for all $I,\, J\in \cali{P}_{s,t}$, $I\subseteq J$. By setting $p_{s,t}^\sharp=\Delta^\sharp_{s,t}(p_{s,t})$, for all $0<s<t$, we deduce that $\{p_{s,t}^\sharp\}_{0<s<t}$ is a unit of the $C^*$-product system $\cali{A}^\sharp$.\end{observation}
\begin{definition} The unit $\{p_{s,t}^\sharp\}_{0<s<t}$ will be referred to as the unit dilation of $\{p_{s,t}\}_{0<s<t}$.\end{definition}

The last proposition of this section shows that the construction of the inductive dilation is categorical, in the sense that any monomorphism of $C^*$-subproduct systems can be extended to a monomorphisms of their inductive dilations, respecting all the inductive structures considered.
\begin{proposition}\label{categorical theorem0}
Suppose that $\cali{A}=\left(\{\m{A}_{s,t}\}_{0<s<t}, \{\Delta_{r,s,t}\}_{0<r<s<t}\right)$ and $\cali{B}=\left(\{\m{B}_{s,t}\}_{0<s<t}, \{\Gamma_{r,s,t}\}_{0<r<s<t}\right)$ are $C^*$-subproduct systems. Let $\{\theta_{s,t}\}_{0<s<t}$ be a monomorphism, respectively an isomorphism, from $\cali{A}$ to $\cali{B}$. Then there exists a unique monomorphism, respectively isomorphism, $\{\theta_{s,t}^\sharp\}_{0<s<t}$ from $\cali{A}^\sharp=\left(\{\m{A}_{s,t}^\sharp\}_{0<s<t},\,\{\Delta_{r,s,t}^\sharp\}_{0<r<s<t}\right)$ to $\cali{B}^\sharp=\left(\{\m{B}_{s,t}^\sharp\}_{0<s<t},\,\{\Gamma_{r,s,t}^\sharp\}_{0<r<s<t}\right)$
that make the diagram
\begin{eqnarray}\label{Nov17a}
\xymatrix{
\m{A}_{s,t}^\sharp  \ar[r]^-{\theta_{s,t}^\sharp} & \m{B}^\sharp_{s,t} \\
\m{A}_{s,t} \ar[r]_-{\theta_{s,t}}\ar[u]^-{\Delta^\sharp_{s,t}} & \m{B}_{s,t}\ar[u]_-{\Gamma^\sharp_{s,t}} }
\end{eqnarray}
 commute, that is $\theta_{s,t}^\sharp \Delta_{s,t}^\sharp=\Gamma_{s,t}^\sharp \theta_{s,t},$ for all $0<s<t$. 
\end{proposition}

\begin{proof} For any two real numbers $0<s<t$ and any partition $I=\{s=\iota_0<\iota_1<\iota_2<\,\dots<\iota_m<\iota_{m+1}=t\}\in \m{P}_{s,t}$, consider the associated tensor product *-monomorphism $\theta_I:\m{A}_I\to\m{B}_I$,  $
\theta_I=\theta_{\iota_0,\iota_1}\otimes \theta_{\iota_1,\iota_2}\otimes \cdots\otimes \theta_{\iota_m,\iota_{m+1}}.$ It is readily seen that the families $\{\theta_I\}_{I\in\cali{P}_{s,t}}$, $\{\Delta_{I,J}\}$ and $\{\Gamma_{I,J}\}$ are compatible, in the sense that
\begin{eqnarray}\label{rok}
\theta_J\Delta_{I,J}=\Gamma_{I,J}\theta_I,
\end{eqnarray}
 for all $I,\,J\in \cali{P}_{s,t}$, $I\subseteq J$. As a result, there exists a unique *-monomorphism $\theta^\sharp_{s,t}:\m{A}_{s,t}^\sharp\to \m{B}_{s,t}^\sharp$ such that \begin{eqnarray}\label{Nov17b}\theta^\sharp_{s,t}\Delta^\sharp_I=\Gamma^\sharp_I\theta_I,\end{eqnarray} for all $I\in\cali{P}_{s,t}$ and $0<s<t$. Because  \begin{eqnarray*}\Gamma_{r,s,t}^\sharp\theta_{r,t}^\sharp \Delta^\sharp_{I\cup J}&=& \Gamma_{r,s,t}^\sharp\Gamma^\sharp_{I\cup J}\theta_{I\cup J}=\Gamma_I^\sharp\theta_I\otimes \Gamma_J^\sharp\theta_J
 =\theta_{r,s}^\sharp\Delta_I^\sharp\otimes \theta_{s,t}^\sharp\Delta_J^\sharp
\\&=&(\theta_{r,s}^\sharp\otimes \theta_{s,t}^\sharp) \Delta_{r,s,t}^\sharp \Delta^\sharp_{I\cup J},\end{eqnarray*} for all real numbers $0<r<s<t$ and all partitions $I\in\cali{P}_{r,s}$, $J\in  \cali{P}_{s,t}$, we deduce that 
  $\{\theta_{s,t}^\sharp\}_{0<s<t}$ is a monomorphism of tensorial $C^*$-systems. It is clear that $\{\theta_{s,t}^\sharp\}_{0<s<t}$ is an isomorphism if $\{\theta_{s,t}\}_{0<s<t}$ is itself an isomorphism.
  
Regarding uniqueness, suppose that $\{\eta_{s,t}^\sharp\}_{0<s<t}$ is another monomorphism from $\cali{A}^\sharp$ to $\cali{B}^\sharp$ that satisfies (\ref{Nov17a}). We will show that $\{\eta_{s,t}^\sharp\}_{0<s<t}$ must satisfy (\ref{Nov17b}) and, consequently, this leads to $\theta^\sharp_{s,t}=\eta^\sharp_{s,t}$, for all $0<s<t$. It suffices to show that  $\{\eta_{s,t}^\sharp\}_{0<s<t}$ satisfies (\ref{Nov17b}) for any partition of the form $I=\{s,\iota,t\}$, where $0<s<\iota<t$. For this purpose, we write $I=\{s,\iota\}\cup\{\iota, t\}$ and use (\ref{compas}) to obtain that \begin{eqnarray*}
  \eta^\sharp _{s,t}\Delta^\sharp _I&=&\left( \Gamma_{s,\iota,t}^\sharp\right)^{-1}\left(\eta_{s,\iota}^\sharp\otimes  \eta_{\iota,t}^\sharp\right)\Delta_{s,\iota, t}^\sharp\Delta^\sharp_{\{s,\iota\}\cup\{\iota, t\}}\\&=&\left( \Gamma_{s,\iota,t}^\sharp\right)^{-1}\left(\eta_{s,\iota}^\sharp\Delta^\sharp_{s,\iota}\otimes  \eta_{\iota,t}^\sharp\Delta_{\iota, t}^\sharp\right)\\&=&\left( \Gamma_{s,\iota,t}^\sharp\right)^{-1}\left(\Gamma_{s,\iota}^\sharp \theta_{s,\iota}\otimes  \Gamma_{\iota,t}^\sharp \theta_{\iota,t}\right)=\Gamma^\sharp_I\theta_I.\end{eqnarray*}
 The proof is now complete.
\end{proof}
We conclude this section with a concrete description of the inductive dilation of a $C^*$-subproduct system of commutative unital $C^*$-algebras.
\begin{example}\label{examp11}
Consider, as in  Example \ref{examp1}, a two-parameter multiplicative system of compact Hausdorff spaces $\left(\{X_{s,t}\}_{0<s<t}, \{\chi_{r,s,t}\}_{0<r<s<t} \right)$ chosen so that every function $\chi_{r,s,t}$ is surjective. For any two positive real numbers $0<s<t$ and any partition $I=\{s=\iota_0<\iota_1<\iota_2<\,\dots<\iota_m<\iota_{m+1}=t\}\in \cali{P}_{s,t}$, let $X_I=
X_{\iota_0, \iota_1}\times X_{\iota_1,\iota_2}\times \dots\times X_{\iota_m,\iota_{m+1}}$ be the cartesian product of the spaces  $X_{\iota_k,\iota_{k+1}}$, and  $\chi_{\{s,t\},I}:X_{I}\to X_{s,t}$ be the continuous surjection defined iteratively according to the model used in Definition \ref{Jan21}, i.e.,
$$\chi_{\{s,t\},I}=\left\{\begin{array}{llll}\chi_{\iota_0,\iota_1,\iota_2}, &\mbox{if}\;m=1\\\chi_{\iota_0,\iota_m,\iota_{m+1}}\left(\chi_{\{\iota_0,\iota_m\},I\setminus\{\iota_{m+1}\}}\times \op{id}_{\iota_m,\iota_{m+1}}\right),
& \mbox{if}\; m\geq 2
\end{array}\right.$$ 
If $J\in \cali{P}_{s,t}$ is a  refinement of $I$, decomposed as $J=I_0\cup \ldots \cup I_m$ with respect to $I$, then let $\chi_{I,J}:X_J\rightarrow X_I$ be the continuous surjection \begin{eqnarray}\label{June16}\chi_{I,J}=\chi_{\{\iota_0,\iota_1\}, I_0}\times \chi_{\{\iota_1,\iota_2\}, I_1}\times \cdots \chi_{\{\iota_m,\iota_{m+1}\}, I_m}.\end{eqnarray}
As in Proposition \ref{lemma inductive limit-c}, it can be shown that the system $
\Bigl\{(X_I,\chi_{I,J})\,|\, I,\,J \in \cali{P}_{s,t},\; I\subseteq J \Bigr\}
$ is a projective system of compact Hausdorff spaces over the partially ordered set $(\cali{P}_{s,t}, \subseteq)$, for all $0<s<t$. Its projective limit  $X_{s,t}^\sharp=\limproj \Bigl\{(X_I,\chi_{I,J})\,|\,I,\,J \in \cali{P}_{s,t},\; I\subseteq J\Bigr\}$ is a non-empty compact Hausdorff space.

Reasoning as in the proof of Theorem \ref{star-isomorphism theorem}, for any real numbers $0<r<s<t$, one can find a homeomorphism $\chi_{r,s,t}^\sharp: X_{r,s}^\sharp\times X_{s,t}^\sharp\to X_{r,t}^\sharp$ uniquely determined by the condition $(\chi_{I\cup J}^\sharp)\comp (\chi_{r,s,t}^\sharp)=\chi_I^\sharp\times \chi_J^\sharp$
for all $I\in\cali{P}_{r,s}$, $J\in  \cali{P}_{s,t}$. The system $\left(\{X_{s,t}^\sharp\}_{0<s<t},\,\{\chi_{r,s,t}^\sharp\}_{0<r<s<t}\right)$ is a two-parameter multiplicative system of compact Hausdorff spaces, and the associate system of commutative $C^*$-algebras $\cali{A}_{\op{com}}^\sharp=\left(\{C(X_{s,t}^\sharp)\}_{0<s<t},\,\{\Delta_{r,s,t}^\sharp\}_{0<r<s<t}\right)$, where $\Delta_{r,s,t}^\sharp$ is the *-homomorphism induced by $\chi_{r,s,t}^\sharp$, is a $C^*$-product system that is isomorphic to the inductive dilation of the unital $C^*$-subproduct system $\cali{A}_{\op{com}}=\left(\{C(X_{s,t})\}_{0<s<t}, \{\Delta_{r,s,t}\}_{0<r<s<t}\right)$ of $\left(\{X_{s,t}\}_{0<s<t}, \{\chi_{r,s,t}\}_{0<r<s<t} \right)$.
\end{example}

\section{The $C^*$-algebra of a unital $C^*$-subproduct system}\label{sec4.1}
In this section, we consider the partially ordered set $\cali{P}=\bigcup_{0<s<t}\cali{P}_{s,t}$ of all finite ordered subsets $I$ of $(0,\infty)$, $|I|\geq 2$, ordered by inclusion. Similar to decomposition (\ref{decoopa}), for any partition $I\in \cali{P}$ of the form $I=\{s=\iota_0<\iota_1<\iota_2<\,\dots<\iota_{m+1}=t\}$, and any refinement $J\in\cali{P}$ of $I$, one can decompose $J$ as \begin{eqnarray}\label{decoop}J=\underline{I}\cup I_0\cup\cdots\cup I_m\cup \overline{I}=\underline{I}\cup I^{\times}\cup \overline{I}\end{eqnarray} where the terminal partitions $\underline{I}$ and $\overline{I}$ are given by $\underline{I}=\{j \in J: j \leq \iota_0\}$, $\overline{I}=\{j\in J: \iota_{m+1}\leq j\},$ and the middle partition $I^{\times}$ is defined as  $I^{\times}= I_1\cup\cdots\cup I_m\in\cali{P}_{s,t}$, where 
for each $0\leq i\leq m$, $I_i=\{j\in J, \iota_i\leq j\leq \iota_{i+1}\}=\{\iota_i=\iota_{i_0}<\iota_{i_1}<\dots<\iota_{i_{n_{I_i}}}<\iota_{i+1}\}\in \cali{P}_{\iota_i,\iota_{i+1}},$ for some integer $n_{I_i}\in\mathbb{N}$ depending on the set $I_i$. We note, on this occasion, that the partition $I^{\times}$ is a refinement of $I.$

If $\cali{A}=\left(\{\m{A}_{s,t}\}_{0<s<t}, \{\Delta_{r,s,t}\}_{0<r<s<t}\right)$ is a unital $C^*$-subproduct system with unit $\{p_{s,t}\}_{0<s<t}$, then the partition decomposition (\ref{decoop}) leads to the corresponding tensorial decomposition of $C^*$-algebras $$
\m{A}_J=\m{A}_{\underline{I}}\otimes \m{A}_{I^{\times}}\otimes \m{A}_{\overline{I}}=\m{A}_{\underline{I}}\otimes \m{A}_{I_0}\otimes \m{A}_{I_1}\otimes\dots\otimes  \m{A}_{I_m}\otimes \m{A}_{\overline{I}},$$ and of projections $$
p_J=p_{\underline{I}}\otimes p_{I^{\times}}\otimes p_{\overline{I}}$$
for all $I\subseteq J\in\cali{P}$, where $p_I\in\m{A}_I$ are defined as in Observation \ref{hasmas}. Note that if $I_0$, or $I_m$, contains just one point (the end point), they will be ignored. 

\begin{proposition}\label{lemma inductive limit} Let $\cali{A}=\left(\{\m{A}_{s,t}\}_{0<s<t}, \{\Delta_{r,s,t}\}_{0<r<s<t}\right)$ be a unital $C^*$-subproduct system and $\{p_{s,t}\}_{0<s<t}$ be a unit of $\cali{A}$.  For any two partitions $I,\,J\in\cali{P}$, $I\subseteq J$, consider the *-monomorphism $\Delta_{I,J}^\times: \m{A}_I\to\m{A}_J$, 
\begin{eqnarray}\Delta_{I,J}^\times (x)=\left\{\begin{array}{ll}\Delta_{I,J}(x), &\mbox{if}\;I,J\in \cali{P}_{s,t},\;s,t>0\\
{p}_{\underline{I}}\otimes \Delta_{I,I^{\times}}(x) \otimes {p}_{\overline{I}},
 &\mbox{otherwise}\end{array}\right.
\end{eqnarray} for all $x\in \m{A}_I.$
Then the system
\[
\Bigl\{(\m{A}_I,\Delta_{I,J}^\times)\,|\,I,\,J \in \cali{P},\; I\subseteq J\Bigr\}
\]
is an inductive system of $C^*$-algebras over the partially ordered set $(\cali{P}, \subseteq)$.
\end{proposition}

\begin{proof} We will show that the compatibility condition
\begin{eqnarray}\label{indo}\Delta_{I,K}^\times=\Delta_{J,K}^\times \Delta_{I,J}^\times,\end{eqnarray} is satisfied for all partitions  $I,\,J,\,K\in\cali{P}$, $I\subseteq J\subseteq K$. For this purpose, we first show that (\ref{indo}) is satisfied when $J,\,K\in \cali{P}_{u,v}$ and $I\in \cali{P}_{s,t}$, where $0<u\leq s<t\leq v$. It is clear that  (\ref{indo}) holds true if $u=s$, or $t=v$, or $u=s$ and $t=v$, so we can assume that $0<u< s<t< v$. Let $J=\underline{I}_J\cup I^{\times}_J\cup \overline{I}_J$ be the decomposition of $J$ with respect to $I\subseteq J$, as in (\ref{decoop}), and $K=\underline{I}_K\cup I^{\times}_K\cup \overline{I}_K$ be the decomposition  of $K$ with respect to $I\subseteq K$. We notice that $\underline{I}_J\subseteq \underline{I}_K,$ $\overline{I}_J\subseteq \overline{I}_K$, and $I\subseteq I^{\times}_J\subseteq I^{\times}_K.$ Using Lemma \ref{Aug17} (ii), we deduce that $\Delta_{J,K}=\Delta_{\underline{I}_J, \underline{I}_K}\otimes \Delta_{ I^{\times}_J,  I^{\times}_K}\otimes \Delta_{\overline{I}_J,  \overline{I}_K}.$
Consequently, for any $x\in\m{A}_I$, we have \begin{eqnarray*}\Delta^\times_{J,K}\Delta^\times_{I,J}(x)&=&\Delta_{J,K}\left({p}_{\underline{I}_J}\otimes \Delta_{I,I^{\times}_J}(x)\otimes {p}_{\overline{I}_J}\right)
\\&=&\Delta_{\underline{I}_J, \underline{I}_K}({p}_{\underline{I}_J})\otimes \Delta_{ I^{\times}_J,  I^{\times}_K}\Delta_{I,I^{\times}_J}(x)\otimes \Delta_{\overline{I}_J,  \overline{I}_K}({p}_{\overline{I}_J})
=\Delta^\times_{I,K}(x),\end{eqnarray*} as required.

We can now prove (\ref{indo}) in full generality. For this, let $I\in \cali{P}_{s,t}$ and $J\in \cali{P}_{u,v}$, where $0<u\leq s<t\leq v$. Since (\ref{indo}) can be easily checked when $u=s$, or $t=v$, or $u=s$ and $t=v$, we assume that $0<u<s<t<v$. In this case, we consider the decomposition  $J=\underline{I}_J\cup I_J^{\times}\cup \overline{I}_J$ of $J$ with respect to $I\subseteq J$, the decomposition $K=\underline{J}_K\cup J_K^{\times}\cup \overline{J}_K$ of $K$ with respect to $J\subseteq K$, as well the decomposition $K=\underline{I}_K\cup I_K^{\times}\cup \overline{I}_K$ of $K$ with respect to $I\subseteq K$. We notice that $I,\,I_J^{\times},\, I_K^{\times}\in \cali{P}_{s,t}$ satisfy $I\subseteq I_J^{\times}\subseteq I_K^{\times}$, while $J,\, J_K^{\times}\in \cali{P}_{u,v}$ satisfy $J\subseteq J_K^{\times}.$ In particular, $I\subseteq J_K^{\times}$, and the partitions in the corresponding decomposition $J_K^{\times}= \underline{I}_{ J_K^{\times}}\cup I^{\times}_{ J_K^{\times}}\cup \overline{I}_{ J_K^{\times}}$ are simply given by $\underline{I}_{ J_K^{\times}}=K\cap[u,s],$ 
$I^{\times}_{ J_K^{\times}}=I^{\times}_K,$ and $\overline{I}_{ J_K^{\times}}=K\cap[t,v].$
Consequently, $\underline{I}_K= \underline{J}_K\cup \underline{I}_{ J_K^{\times}}$ and $\overline{I}_K=\overline{J}_K\cup \overline{I}_{ J_K^{\times}}.$ Using all this, we have
\begin{eqnarray*}\Delta_{J,K}^\times\Delta^\times_{I,J}(x)&=&
={p}_{\underline{J}_K}\otimes \Delta_{J,J_K^{\times}}^\times\Delta^\times_{I,J}(x)\otimes {p}_{\overline{J}_K}
\\&=&p_{\underline{J}_K}\otimes \Delta^\times_{I,J_K^{\times}}(x)\otimes {p}_{\overline{J}_K}
=p_{\underline{I}_K}\otimes \Delta_{I,I_K^{\times}}(x)\otimes {p}_{\overline{I}_K}\\&=&\Delta^\times_{I,K}(x),
\end{eqnarray*}
for all $x\in\m{A}_I$. The proposition is proved. 
\end{proof}

\begin{definition} \label{deff} The $C^*$-algebra $C^*_{\{ p_{s,t} \}}(\cali{A})$ of a unital $C^*$-subproduct system  $\cali{A}=\left(\{\m{A}_{s,t}\}_{0<s<t}, \{\Delta_{r,s,t}\}_{0<r<s<t}\right)$ with respect to a unit $\{p_{s,t}\}_{0<s<t}$ is defined as the $C^*$-inductive limit
  \begin{eqnarray*}C^*_{\{ p_{s,t} \}}(\cali{A})=\limind\, \Bigl\{(\m{A}_I,\Delta^\times_{I,J})\,|\,I,\,J\in\cali{P},\; I\subseteq J\Bigr\}\end{eqnarray*}
of the inductive system $\{(\m{A}_I,\Delta^\times_{I,J})\,|\, ,I,\,J\in\cali{P},\; I\subseteq J\}$ constructed above.
  \end{definition}
   Let $\{\Delta^\times_I\}_{I\in\cali{P}}$ be the family of connecting *-monomorphisms $\Delta _I^\times:\m{A}_I\to C^*_{\{p_{s,t} \}}(\cali{A})$ associated with the inductive limit construction. If $I$ has only two elements, say $I=\{s,t\}$ where $s<t$, then we shall simply write $\Delta_{s,t}^\times$ instead of $\Delta_{\{s,t\}}^\times$. Both $\Delta^\times_{I,J}$ and $\Delta^\times_I$ depend on $\{p_{s,t}\}_{0<s<t}$ but, for ease of notation, we do not write this explicitly. 

We note that if $\m{A}_{s,t}$ are all unital $C^*$-algebras and $\{p_{s,t}\}_{0<s<t}$ is the trivial unit of $\cali{A}$, i.e.,  $p_{s,t}=1_{\m{A}_{s,t}}$, for all $0<s<t$, then the associated $C^*$-algebra $C^*_1(\cali{A})$ is quasi-local, in the sense of \cite{Haag,BR}.

The following proposition shows that any monomorphism of unital $C^*$-subproduct systems induces a *-monomorphism between their associate $C^*$-algebras.
\begin{proposition}\label{categorical theorem}Suppose that $\cali{A}=\left(\{\m{A}_{s,t}\}_{0<s<t}, \{\Delta_{r,s,t}\}_{0<r<s<t}\right)$ and $\cali{B}=\left(\{\m{B}_{s,t}\}_{0<s<t}, \{\Gamma_{r,s,t}\}_{0<r<s<t}\right)$ are unital $C^*$-subproduct systems with units $\{p_{s,t}\}_{0<s<t}$ and $\{q_{s,t}\}_{0<s<t}$ respectively. Then for any monomorphism $\{\theta_{s,t}\}_{0<s<t}$ from $\cali{A}$ to $\cali{B}$ satisfying $\theta_{s,t}(p_{s,t})=q_{s,t}$, for all $0<s<t$, there exists a unique *-monomorphism of $C^*$-algebras $\theta:C^*_{\{p_{s,t}\}}(\cali{A})\to C^*_{\{q_{s,t}\}}(\cali{B})$ such that \begin{eqnarray}\label{Nov12a}\theta \Delta_{s,t}^\times=\Gamma_{s,t}^\times \theta_{s,t},\end{eqnarray} for all $0<s<t$.
\end{proposition}
\begin{proof}
For any partition $I\in \cali{P}$ of the form $I=\{\iota_0<\iota_1<\ldots <\iota_{m+1}\}$, we consider the tensor product $\theta_I=\theta_{\iota_0,\iota_1}\otimes \theta_{\iota_1,\iota_2}\otimes \cdots \otimes\theta_{\iota_m,\iota_{m+1}}$, from $\m{A}_I$ to $\m{B}_I.$ We show that the resulting family of *-monomorphisms $\{\theta_I\}_{I\in\cali{P}}$ satisfies the compatibility condition
 \begin{eqnarray}\label{franc}\theta_J\Delta^\times_{I,J}=\Gamma^\times_{I,J}\theta_I,\end{eqnarray} for all $I,\, J\in\cali{P}$, $I\subseteq J$. 
 
 First of all, we will show that (\ref{franc}) is satisfied when the partitions $I$ and $J$ have the same endpoints. Let then $I,\,J\in\cali{P}_{s,t}$, $I\subseteq J$, where $I=\{s=\iota_0<\ldots <\iota_{m+1}=t\}$ and $J=\{s=j_0<j_1<\ldots <j_{n+1}=t\}$. Because $\theta_J\Delta^\times_{I,J}=\big(\theta_{I_0}\Delta_{\{\iota_0,\iota_1\},I_0}\big)\otimes \cdots \otimes \big(\theta_{I_m} \Delta_{\{\iota_m,\iota_{m+1}\},I_m}\big)$, where $J= I_0\cup\cdots\cup I_m$ is the decomposition of $J$ with respect to $I$, we can assume, without loss of generality, that $I$ is the trivial partition $\{s,t\}$.
Noticing that \begin{eqnarray*}&&\theta_J\left(\Delta_{j_0,j_1,j_2}\otimes \op{id}_{\m{A}_{j_2, j_3}}\otimes\cdots \otimes \op{id}_{\m{A}_{j_n, j_{n+1}}}\right)\\&=&(\theta_{j_0,j_1}\otimes \theta_{j_1,j_2})\Delta_{j_0,j_1,j_2}\otimes \theta_{J\setminus \{j_0,j_1\}}\\&=&\Gamma_{j_0,j_1,j_2} \theta_{j_0,j_2}\otimes \theta_{J\setminus \{j_0,j_1\}}=
\left(\Gamma_{j_0,j_1,j_2}\otimes \op{id}_{\m{B}_{j_2, j_3}}\otimes\cdots \otimes \op{id}_{\m{B}_{j_n, j_{n+1}}}\right)\theta_{J\setminus\{j_1\}},
\end{eqnarray*}
and iterating this calculation, we obtain that
\begin{eqnarray*}
&&\theta_J\Delta_{\{s,t\},J}^\times=\theta_J\left(\Delta_{j_0,j_1,j_2}\otimes \op{id}\big)\cdots\big(\Delta_{j_0,j_{n-1},j_n}\otimes \op{id}\right) \Delta_{j_0,j_n,j_{n+1}}
\\&=&\left(\Gamma_{j_0,j_1,j_2}\otimes \op{id}\right)\theta_{J\setminus\{j_1\}}\left(\Delta_{j_0,j_2,j_3}\otimes \op{id}\big)\cdots\big(\Delta_{j_0,j_{n-1},j_n}\otimes \op{id}\right) \Delta_{j_0,j_n,j_{n+1}}
\\
&=&
\left(\Gamma_{j_0,j_1,j_2}\otimes \op{id}\right)\left(\Gamma_{j_0,j_2,j_3}\otimes \op{id}\right)\theta_{J\setminus \{j_1,j_2\}})
\left(\Delta_{j_0,j_3,j_4}\otimes \op{id}\right)\cdots\\&&\cdots 
\left(\Delta_{j_0,j_{n-1},j_n}\otimes \op{id}\right) \Delta_{j_0,j_n,j_{n+1}}
=\cdots\cdots\cdots
\\&=&\left(\Gamma_{j_0,j_1,j_2}\otimes \op{id}\right)\left(\Gamma_{j_0,j_2,j_3}\otimes \op{id}\right)\cdots \left(\Gamma_{\iota_0,j_{n-1},j_n}\otimes \op{id}\right)\Gamma_{j_0,j_n,j_{n+1}}\theta_{\{j_0,j_{n+1}\}}
\\&=&\Gamma_{\{s,t\},J}^\times\theta_{\{s,t\}},
\end{eqnarray*}
as required.

In the general case of two arbitrary partitions $I,\, J\in\cali{P}$, $I\subseteq J$, consider the decomposition $J=\underline{I}\cup I^{\times}\cup \overline{I}$ of $J$ with respect to $I$. Then for any $x\in\m{A}_I$, we have that
\begin{eqnarray*}\theta_J \Delta^\times_{I,J}(x)={q}_{{\underline{I}}}\otimes \theta_{I^{\times}}\Delta_{I,I^{\times}}(x)\otimes {q}_{{\overline{I}}}={q}_{{\underline{I}}}\otimes \Gamma_{I,I^{\times}}\theta_{I}(x)\otimes {q}_{{\overline{I}}}
=\Gamma^\times_{I,J} \theta_I(x),\end{eqnarray*}
thus showing (\ref{franc}). Using (\ref{franc}) and the universal property of the inductive limit, we deduce that there is a unique *-monomorphism $\theta:C^*_{\{p_{s,t}\}}(\cali{A})\to C^*_{\{q_{s,t}\}}(\cali{B})$ that satisfies the compatibility condition \begin{eqnarray}\label{Nov12}
\theta \Delta_I^\times=\Gamma_I^\times \theta_I,\end{eqnarray} for all $I\in\cali{P}$. The uniqueness of $\theta$ can be shown as in the proof of Proposition \ref{categorical theorem0}. 
\end{proof}
The following corollary can be immediately deduced from Proposition \ref{categorical theorem}. We record it here for future reference.
\begin{corollary}\label{cucuss} Let $\cali{A}=\left(\{\m{A}_{s,t}\}_{0<s<t}, \{\Delta_{r,s,t}\}_{0<r<s<t}\right)$ be a unital $C^{\star}$-subproduct system with unit $\{p_{s,t}\}_{0<s<t}$. Then there exists a unique *-monomorphism $\Upsilon:C^*_{\{p_{s,t}\}}(\cali{A})\to C^*_{\{p_{s,t}^\sharp\}}(\cali{A}^\sharp)$ such that $$\Upsilon \Delta_{s,t}^\times=\big(\Delta_{s,t}^\sharp\big)^\times \Delta_{s,t}^\sharp,$$ for all $0<s<t$, where $\cali{A}^\sharp=\left(\{\m{A}_{s,t}^\sharp\}_{0<s<t},\,\{\Delta_{r,s,t}^\sharp\}_{0<r<s<t}\right)$ is the inductive dilation of $\cali{A}$.
\end{corollary}


\section{A one-parameter co-multiplication on $C^*_{\{p_{s,t}\}}(\cali{A})$}\label{sec5+.1}

We start this section by showing that the constituent $C^*$-algebras $\m{A}_{s,t}^\sharp$ of the inductive dilation $\cali{A}^\sharp=\left(\{\m{A}_{s,t}^\sharp\}_{0<s<t},\,\{\Delta_{r,s,t}^\sharp\}_{0<r<s<t}\right)$ of a unital $C^*$-subproduct system $\cali{A}=\left(\{\m{A}_{s,t}\}_{0<s<t}, \{\Delta_{r,s,t}\}_{0<r<s<t}\right)$ with unit $\{p_{s,t}\}_{0<s<t}$ can be assembled into an inductive system over the partially ordered set $(I_{(0,\infty)},\subseteq)$ of all open intervals of positive real numbers, ordered by inclusion. We note that the same procedure can be applied to any $C^*$-product system, and not only to inductive dilations.

\begin{proposition}\label{Nov06}
Let $\cali{A}=\left(\{\m{A}_{s,t}\}_{0<s<t}, \{\Delta_{r,s,t}\}_{0<r<s<t}\right)$ be a unital $C^*$-subproduct system and $\cali{A}^\sharp=\left(\{\m{A}_{s,t}^\sharp\}_{0<s<t},\,\{\Delta_{r,s,t}^\sharp\}_{0<r<s<t}\right)$ be its inductive dilation. Let $\{p_{s,t}\}_{0<s<t}$ be a unit of $\cali{A}$, and $\{p_{s,t}^\sharp\}_{0<s<t}$ be the unit dilation of $\{p_{s,t}\}_{0<s<t}$. For any real numbers $0<u\leq s<t\leq v$, consider the  *-monomorphism $\Delta^\sharp _{(s,t),(u,v)}: \m{A}_{s,t}^\sharp\to \m{A}_{u,v}^\sharp$, acting as  \begin{eqnarray*}\Delta^\sharp _{(s,t),(u,v)}(x)=\left\{\begin{array}{ll}
x,&\mbox{if}\;u=s<t= v\\
(\Delta_{s,t,v}^\sharp)^{-1}(x \otimes {p}^\sharp_{t,v} ),&\mbox{if}\;\;u=s<t< v\\
(\Delta_{u,s,t}^\sharp)^{-1}({p}^\sharp_{u,s}\otimes x ),&\mbox{if}\;\;u<s<t=v\\
(\Delta^\sharp_{u,s,v})^{-1}\left(\op{id}_{\m{A}_{u,s}^\sharp}\otimes \Delta^\sharp_{s,t,v}\right)^{-1}\left({p}^\sharp_{u,s}\otimes  x \otimes {p}^\sharp_{t,v} \right),&\mbox{if}\;\;u< s<t< v \end{array}\right. \end{eqnarray*}
for every $x\in \m{A}_{s,t}^\sharp$. Then the system $
\left\{ \left(\m{A}_{s,t}^\sharp, \Delta^\sharp _{(s,t),(u,v)}\right)\,|\, (s,t)\in I_{(0,\infty)}\right\}$
is an inductive system of $C^*$-algebras over the partially ordered set $(I_{(0,\infty)},\subseteq)$.
\end{proposition}
\begin{proof}
We will show that the family of *-monorphisms defined above satisfies the compatibility condition \begin{eqnarray*}\Delta^\sharp _{(s,t),(u,v)} \Delta^\sharp _{(q,r),(s,t)}= \Delta^\sharp _{(q,r),(u,v)},\end{eqnarray*} for any strict inclusion of open intervals $(q,r)\subset (s,t)\subset (u,v)$. First of all, we notice that
\begin{eqnarray*}&&\Delta^\sharp _{(s,t),(u,v)} \Delta^\sharp _{(q,r),(s,t)}(x)=(\Delta^\sharp_{u,s,v})^{-1}\left(\op{id}_{\m{A}_{u,s}^\sharp}\otimes \Delta^\sharp_{s,t,v}\right)^{-1}
\left(\op{id}_{\m{A}^\sharp_{u,s}}\otimes   \Delta^\sharp_{s,q,t} \otimes \op{id}_{\m{A}^\sharp_{t,v}} \right)^{-1}\\&&
\left(\op{id}_{\m{A}^\sharp_{u,s}}\otimes  \op{id}_{\m{A}_{s,q}^\sharp}\otimes \Delta^\sharp_{q,r,t} \otimes \op{id}_{\m{A}^\sharp_{t,v}} \right)^{-1}
\left(\Delta^\sharp_{u,s,q}\otimes \op{id}_{\m{A}^\sharp_{q,r}} \otimes \Delta^\sharp_{r,t,v} \right)
\left({p}^\sharp_{u,q}\otimes  x \otimes {p}^\sharp_{r,v} \right),
\end{eqnarray*}
for all  $x\in \m{A}_{q,r}^\sharp.$ In addition, using the co-associativity laws $\left(\Delta^\sharp_{s,q,t} \otimes \op{id}_{\m{A}^\sharp_{t,v}}\right)\Delta^\sharp_{s,t,v}= \left( \op{id}_{\m{A}^\sharp_{s,q}} \otimes\Delta^\sharp_{q,t,v} \right)\Delta^\sharp_{s,q,v}$ and $ \left(\Delta^\sharp_{q,r,t}\otimes \op{id}_{\m{A}^\sharp_{t,v}}\right)\Delta^\sharp_{q,t,v}= \left(\op{id}_{\m{A}^\sharp_{q,r}}\otimes
\Delta^\sharp_{r,t,v} \right)\Delta^\sharp_{q,r,v} $, and then canceling out a few factors, we turn the above identity  into 
\begin{eqnarray*}&&\Delta^\sharp _{(s,t),(u,v)} \Delta^\sharp _{(q,r),(s,t)}(x)=(\Delta^\sharp_{u,s,v})^{-1}
\left(\op{id}_{\m{A}^\sharp_{u,s}} \otimes \Delta^\sharp_{s,q,v} \right)^{-1}
\left(\op{id}_{\m{A}^\sharp_{u,s}}\otimes  \op{id}_{\m{A}_{s,q}^\sharp}\otimes \Delta^\sharp_{q,r,v}  \right)^{-1}
\\&&
\left(\Delta^\sharp_{u,s,q}\otimes \op{id}_{\m{A}^\sharp_{q,r}} \otimes  \op{id}_{\m{A}^\sharp_{r,v}} \right)
\left({p}^\sharp_{u,q}\otimes  x \otimes {p}^\sharp_{r,v} \right).
%
\end{eqnarray*}
Finally, writing the first term of the above expression as $$(\Delta^\sharp_{u,s,v})^{-1}=(\Delta^\sharp_{u,q,v})^{-1}\left(\Delta^\sharp_{u,s,q}\otimes \op{id}_{\m{A}_{q,v}^\sharp}
\right)^{-1}\left(\op{id}_{\m{A}_{u,s}^\sharp}\otimes \Delta^\sharp_{s,q,v}\right),$$ we obtain
\begin{eqnarray*}&&\Delta^\sharp _{(s,t),(u,v)} \Delta^\sharp _{(q,r),(s,t)}(x)
=(\Delta^\sharp_{u,q,v})^{-1}\left(\Delta^\sharp_{u,s,q}\otimes \op{id}_{\m{A}_{q,v}^\sharp}
\right)^{-1}
\left(\op{id}_{\m{A}^\sharp_{u,s}}\otimes  \op{id}_{\m{A}_{s,q}^\sharp}\otimes \Delta^\sharp_{q,r,v}  \right)^{-1}
\\&&
\left(\Delta^\sharp_{u,s,q}\otimes \op{id}_{\m{A}^\sharp_{q,r}} \otimes  \op{id}_{\m{A}^\sharp_{r,v}} \right)
\left({p}^\sharp_{u,q}\otimes  x \otimes {p}^\sharp_{r,v} \right)\\
&=&(\Delta^\sharp_{u,q,v})^{-1}\left(\op{id}_{\m{A}^\sharp_{u,q}}\otimes \Delta^\sharp_{q,r,v}  \right)^{-1}
\left({p}^\sharp_{u,q}\otimes  x \otimes {p}^\sharp_{r,v} \right)=\Delta_{(q,r),(u,v)}^\sharp (x),
\end{eqnarray*}
as needed. The remaining cases, i.e., those of non-strict inclusion of open intervals, can also be easily shown. 
\end{proof}
\begin{definition}\label{Oct11}
Let $\cali{A}=\left(\{\m{A}_{s,t}\}_{0<s<t}, \{\Delta_{r,s,t}\}_{0<r<s<t}\right)$ be a unital $C^*$-subproduct system with unit $\{p_{s,t}\}_{0<s<t}$. We define the $C^*$-algebra $\m{A}^\diamond_{\{p_{s,t}\}}$ as the $C^*$-inductive limit of the inductive system constructed above, i.e., 
$$
\m{A}^\diamond_{\{p_{s,t}\}}=\limind
\left\{ \left(\m{A}_{s,t}^\sharp, \Delta^\sharp _{(s,t),(u,v)}\right)\,|\, (s,t)\in I_{(0,\infty)}\right\},$$
 with connecting *-monomorphisms $
 \Delta^{\diamond}_{s,t}:\m{A}_{s,t}^\sharp\to \m{A}^\diamond_{\{p_{s,t}\}}$, for all $0<s<t$.
\end{definition}
The next corollary is a direct consequence of Proposition \ref{categorical theorem0} and the universal property of the inductive limit.
\begin{corollary}Suppose that $\cali{A}=\left(\{\m{A}_{s,t}\}_{0<s<t}, \{\Delta_{r,s,t}\}_{0<r<s<t}\right)$ and $\cali{B}=\left(\{\m{B}_{s,t}\}_{0<s<t}, \{\Gamma_{r,s,t}\}_{0<r<s<t}\right)$ are unital $C^*$-subproduct systems, and let $\cali{A}^\sharp=\left(\{\m{A}_{s,t}^\sharp\}_{0<s<t},\,\{\Delta_{r,s,t}^\sharp\}_{0<r<s<t}\right)$, resp. $\cali{B}^\sharp=\left(\{\m{B}_{s,t}^\sharp\}_{0<s<t},\,\{\Gamma_{r,s,t}^\sharp\}_{0<r<s<t}\right)$, be their inductive dilations. Let  $\{p_{s,t}\}_{0<s<t}$ and $\{q_{s,t}\}_{0<s<t}$ be units of $\cali{A}$ and $\cali{B}$ respectively. If $\{\theta_{s,t}\}_{0<s<t}$ is a monomorphism, respectively isomorphism, of $C^*$-subproduct systems from $\cali{A}$ to $\cali{B}$ satisfying $\theta_{s,t}(p_{s,t})=q_{s,t}$, for all $0<s<t$, then there exists a unique *-monomorphism  $\theta^{\diamond}$, respectively *-isomorphism, of $C^*$-algebras from $\m{A}^{\diamond}_{\{p_{s,t}\}}$ to $\m{B}^{\diamond}_{\{q_{s,t}\}}$ that makes the diagram
\begin{eqnarray*}
\xymatrix{
\m{B}_{s,t}  \ar[r]^-{\Gamma_{s,t}^\sharp} & \m{B}^\sharp_{s,t} \ar[r]^-{\Gamma_{s,t}^\diamond} & \m{B}^{\diamond}_{\{q_{s,t}\}} \\
\m{A}_{s,t} \ar[r]_-{\Delta_{s,t}^\sharp}\ar[u]^-{\theta_{s,t}} & \m{A}_{s,t}^\sharp\ar[u]_-{\theta^\sharp_{s,t}} \ar[r]^-{\Delta_{s,t}^\diamond} &\m{A}^{\diamond}_{\{p_{s,t}\}} \ar[u]_-{\theta^\diamond}}
\end{eqnarray*}
 commutative, for all $0<s<t$.
\end{corollary}
\begin{proof}
It is enough to show that $\theta^\sharp_{u,v}{\Delta_{(s,t), (u,v)}^\sharp}=\Gamma_{(s,t), (u,v)}^\sharp\theta^\sharp_{s,t}$,  for all real numbers $0<u< s<t< v$.
For this, let $x\in \m{A}_{s,t}^\sharp$. We then have \begin{eqnarray*}
&& \theta^\sharp_{u,v}{\Delta_{(s,t), (u,v)}^\sharp}(x)= \theta^\sharp_{u,v}(\Delta^\sharp_{u,s,v})^{-1}\left(\op{id}_{\m{A}_{u,s}^\sharp}\otimes \Delta^\sharp_{s,t,v}\right)^{-1}({p}^\sharp_{u,s}\otimes  x \otimes {p}^\sharp_{t,v} )\\
 &=&(\Gamma^\sharp_{u,s,v})^{-1}( \theta^\sharp_{u,s}\otimes  \theta^\sharp_{s,v}(\Delta^\sharp_{s,t,v})^{-1})({p}^\sharp_{u,s}\otimes  x \otimes {p}^\sharp_{t,v} )\\
 &=&(\Gamma^\sharp_{u,s,v})^{-1}(\op{id}_{\m{B}_{u,s}^\sharp}\otimes \Gamma^\sharp_{s,t,v})^{-1}( \theta^\sharp_{u,s}\otimes \theta^\sharp_{s,t}\otimes  \theta^\sharp_{t,v})({p}^\sharp_{u,s}\otimes  x \otimes {p}^\sharp_{t,v})\\
 &=&\Gamma_{(s,t), (u,v)}^\sharp\theta^\sharp_{s,t}(x),
  \end{eqnarray*}
  as needed.
\end{proof}
The following proposition gives a useful compatibility relation between the connecting *-monomorphism introduced so far: $\Delta^\sharp_I$'s of Definition \ref{Oct5}, $\Delta^\sharp_{(s,t),(u,v)}$'s of Proposition \ref{Nov06}, and $\Delta_{I,J}^\times$'s of Proposition \ref{lemma inductive limit}.
\begin{proposition}\label{March23} Let $\cali{A}=\left(\{\m{A}_{s,t}\}_{0<s<t}, \{\Delta_{r,s,t}\}_{0<r<s<t}\right)$ be a unital $C^*$-subproduct system with unit $\{p_{s,t}\}_{0<s<t}$. For any two partitions $I\in\cali{P}_{s,t}$, $J\in\cali{P}_{u,v}$, $I\subseteq J$, where $0<u\leq s<t\leq v$, the diagram
\begin{eqnarray*}
\xymatrix{
\m{A}_J  \ar[r]^-{\Delta_J^\sharp} & \m{A}^\sharp_{u,v} \\
\m{A}_I \ar[r]_-{\Delta^\sharp_I}\ar[u]^-{\Delta^\times_{I,J}} & \m{A}^\sharp_{s,t}\ar[u]_-{\Delta^\sharp_{(s,t),(u,v)}.} }
\end{eqnarray*}
commutes, that is, $\Delta_J^\sharp\Delta_{I,J}^\times=\Delta^\sharp_{(s,t),(u,v)}\Delta^\sharp_I.$
\end{proposition}
\begin{proof}
We assume that $0<u< s<t< v$; the remaining cases can be treated similarly. Let $J=\underline{I}\cup I^{\times}\cup \overline{I}$ be the decomposition of $J$ with respect to $I$. Using (\ref{compas}) twice, for any $x\in \m{A}_I$ we have that\begin{eqnarray*}&&\Delta_J^\sharp\Delta_{I,J}^\times (x)=
(\Delta^\sharp_{u,s,v})^{-1}\left(\Delta_{\underline{I}}^\sharp \otimes \Delta_{I^{\times}\cup \overline{I}}^\sharp\right)\left({p}_{\underline{I}}\otimes \Delta_{I,I^{\times}}(x) \otimes {p}_{\overline{I}} \right)
\\&=&(\Delta^\sharp_{u,s,v})^{-1}\left(\Delta_{\underline{I}}^\sharp \otimes (\Delta^\sharp_{s,t,v})^{-1}\left(\Delta_{I^{\comp}}^\sharp\otimes \Delta_{ \overline{I}}^\sharp\right)\right)\left({p}_{\underline{I}}\otimes \Delta_{I,I^{\times}}(x) \otimes {p}_{\overline{I}} \right)
\\&=&(\Delta^\sharp_{u,s,v})^{-1}\left(\op{id}_{\m{A}_{u,s}^\sharp}\otimes (\Delta^\sharp_{s,t,v})^{-1} \right)\left(\Delta_{\underline{I}}^\sharp \otimes \Delta_{I^{\times}}^\sharp\otimes \Delta_{ \overline{I}}^\sharp\right)\left({p}_{\underline{I}}\otimes \Delta_{I,I^{\times}}(x) \otimes {p}_{\overline{I}} \right)
\\
&=&(\Delta^\sharp_{u,s,v})^{-1}\left(\op{id}_{\m{A}_{u,s}^\sharp}\otimes \Delta^\sharp_{s,t,v}\right)^{-1}\left({p}^\sharp_{u,s}\otimes \Delta_{I}^\sharp (x) \otimes {p}^\sharp_{t,v} \right)\\&=&\Delta^\sharp_{(s,t),(u,v)}\Delta^\sharp_I(x),
\end{eqnarray*}
which proves our claim.\end{proof}
\begin{corollary}\label{Nov26}Let $\cali{A}=\left(\{\m{A}_{s,t}\}_{0<s<t}, \{\Delta_{r,s,t}\}_{0<r<s<t}\right)$ be a unital $C^*$-subproduct system with unit $\{p_{s,t}\}_{0<s<t}$. For any two partitions $I\in\cali{P}_{s,t}$, $J\in\cali{P}_{u,v}$, $I\subseteq J$, where $0<u\leq s<t\leq v$, the diagram
\begin{eqnarray*}\label{pisubF}
\xymatrix{
\m{A}_J  \ar[r]^-{\Delta_J^\sharp} & \m{A}^\sharp_{u,v} \ar[rd]^-{\Delta_{u,v}^{\diamond}}&\\
&&\m{A}^\diamond_{\{p_{s,t}\}}\\
\m{A}_I \ar[r]_-{\Delta^\sharp_I}\ar[uu]^-{\Delta^{\times}_{I,J}} & \m{A}^\sharp_{s,t}\ar[ru]_-{\Delta^{\diamond}_{s,t}}& }
\end{eqnarray*}
commutes, that is, $\Delta_{s,t}^{\diamond}\Delta_I^\sharp=\Delta^{\diamond}_{u,v}\Delta^\sharp_J\Delta_{I,J}^\times$.
\end{corollary}
\begin{proof}
We have  $\Delta_{s,t}^{\diamond}\Delta_I^\sharp=\Delta_{u,v}^{\diamond}\Delta^\sharp_{(s,t),(u,v)}\Delta_I^\sharp=\Delta_{u,v}^{\diamond}\Delta^\sharp_J\Delta_{I,J}^\times.$
\end{proof}
Using the above results, we can now show that the $C^*$-algebras $C^*_{\{p_{s,t}\}}(\cali{A})$ and $\m{A}^{\diamond}_{\{p_{s,t}\}}$ are isomorphic. 
\begin{theorem}\label{March24}
Let $\cali{A}=\left(\{\m{A}_{s,t}\}_{0<s<t}, \{\Delta_{r,s,t}\}_{0<r<s<t}\right)$ be a unital $C^*$-  
 subproduct system with unit $\{p_{s,t}\}_{0<s<t}$. There exists a *-isomorphism $\delta: C^*_{\{p_{s,t}\}}(\cali{A})\to \m{A}^{\diamond}_{\{p_{s,t}\}}$, uniquely determined by the condition \begin{eqnarray}\label{Nov26a}\delta\Delta_I^\times=\Delta^{\diamond}_{s,t}\Delta^\sharp_I,\end{eqnarray} for all $I\in \cali{P}_{s,t}$ and $0<s<t$.
\end{theorem}
\begin{proof} We deduce from Corollary \ref{Nov26} and the universal property of the inductive limit that there exists a unique *-monomorphism $\delta: C^*_{\{p_{s,t}\}}(\cali{A})\to \m{A}^{\diamond}_{\{p_{s,t}\}}$ that satisfies (\ref{Nov26a}).
To show that $\delta$ is an isomorphism, it is enough to show that the set $\bigcup_{ (s,t)\in I_{(0,\infty)}}\bigcup_{I\in\cali{P}_{s,t}}\Delta^{\diamond}_{s,t}\Delta^\sharp_I\left(\m{A}_I \right)$ is everywhere dense in $ \m{A}^{\diamond}_{\{p_{s,t}\}}$.
For this purpose, let $x\in\m{A}^{\diamond}$ and $\varepsilon>0$ be chosen arbitrarily. Because  $\bigcup_{(s,t)\in I_{(0,\infty)}}\Delta^{\diamond}_{s,t}\left(\m{A}_{s,t}^\sharp \right)$ is everywhere dense in $ \m{A}^{\diamond}_{\{p_{s,t}\}}$, we have $\|x-\Delta^{\diamond}_{s_0,t_0}(y)\|<\frac{\varepsilon}{2}$, for some open interval $(s_0,t_0)\in I_{(0,\infty)}$ and some element $y\in \m{A}_{s_0,t_0}^\sharp$. Similarly, because $\bigcup_{I\in\cali{P}_{s_0,t_0}}\Delta^\sharp_I\left(\m{A}_I \right)$ is everywhere dense in $\m{A}^{\sharp}_{s_0,t_0}$, one can find a partition $I_0\in \cali{P}_{s_0,t_0}$ and an element $z\in \m{A}_{I_0}$ such that $\|\Delta^{\diamond}_{s_0,t_0}(y)-\Delta^{\diamond}_{s_0,t_0}\Delta^\sharp_{I_0}(z)\|< \frac{\varepsilon}{2}$. Consequently, we obtain that $\|x-\Delta^{\diamond}_{s_0,t_0}\Delta^\sharp_{I_0}(z)\|< \varepsilon$, so $\bigcup_{ (s,t)\in I_{(0,\infty)}}\bigcup_{I\in\cali{P}_{s,t}}\Delta^{\diamond}_{s,t}\Delta^\sharp_I\left(\m{A}_I \right)$ is everywhere dense in $ \m{A}^{\diamond}_{\{p_{s,t}\}}$. 
\end{proof}
We are now ready to prove the main result of this section.
\begin{theorem} \label{April2i} Let $\cali{A}=\left(\{\m{A}_{s,t}\}_{0<s<t}, \{\Delta_{r,s,t}\}_{0<r<s<t}\right)$ be a unital $C^*$-  
 subproduct system with unit $\{p_{s,t}\}_{0<s<t}$. There exists a family $\{\Delta_s\}_{s>0}$ of *-monomorphisms $\Delta_s: C^*_{\{p_{s,t}\}}(\cali{A})\to  C^*_{\{p_{s,t}\}}(\cali{A})\otimes  C^*_{\{p_{s,t}\}}(\cali{A})$ that satisfy the identity
\begin{eqnarray}\label{exx}(\Delta_r\otimes\op{id})\Delta_s=(\op{id}\otimes\Delta_s)\Delta_r,\end{eqnarray} for all $0<r<s$.
\end{theorem}
\begin{proof}
First, we construct a family $\{\iota_s\}_{s>0}$ of  *-monomorphism $\iota_s:\m{A}^{\diamond}_{\{p_{s,t}\}}\to  C^*_{\{p_{s,t}\}}(\cali{A})\otimes C^*_{\{p_{s,t}\}}(\cali{A})$ that is compatible with the inductive limit structures of the $C^*$-algebras $\m{A}^{\diamond}_{\{p_{s,t}\}}$ and $C^*_{\{p_{s,t}\}}(\cali{A})\otimes C^*_{\{p_{s,t}\}}(\cali{A})$. For this purpose, we identify the $C^*$-algebra $ C^*_{\{p_{s,t}\}}(\cali{A})\otimes C^*_{\{p_{s,t}\}}(\cali{A})$ with the inductive limit of the $C^*$-inductive system $\{(\m{A}_I\otimes \m{A}_J,\Delta^\times_{I, I'}\otimes \Delta^\times_{J, J'})\,|\, (I, J) \in \cali{P}\times \cali{P}\}$ over the cartesian product $\cali{P}\times \cali{P}$, endowed with the product order, that is
$$C^*_{\{p_{s,t}\}}(\cali{A})\otimes  C^*_{\{p_{s,t}\}}(\cali{A})=\limind\, \Bigl\{(\m{A}_I\otimes \m{A}_J,\Delta^\times_{I, I'}\otimes \Delta^\times_{J, J'})\,|\, (I, J) \in \cali{P}\times \cali{P}\Bigr\}.$$ Furthermore, as in the proof of Theorem \ref{star-isomorphism theorem}, we make the identification $$\m{A}^\sharp_{r,t}=\limind\{(\m{A}_{I}\otimes \m{A}_{J} ,\Delta_{I,I^{\prime}}\otimes \Delta_{J, J^{\prime}} )\,|\, (I,J) \in \cali{P}_{r,s}\times\cali{P}_{s,t}\},$$
for all $0<r<s<t$. Because $\Delta^\times_{I, I'}\otimes \Delta^\times_{J, J'}=\Delta_{I, I'}\otimes \Delta_{J, J'}$, for all $I,\, I'\in  \cali{P}_{r,s}$, $I\subseteq I'$, and $J,\, J'\in  \cali{P}_{s,t}$, $J\subseteq J'$, we infer from the universal property of the inductive limit that there exists a *-monomorphisms $\iota_{r,s,t}:\m{A}^\sharp_{r,t}\to C^*_{\{p_{s,t}\}}(\cali{A})\otimes  C^*_{\{p_{s,t}\}}(\cali{A})$, uniquely determined by the condition
\begin{eqnarray}\label{Nov10a}
\iota_{r,s,t}\Delta^\sharp_{I\cup J}= \Delta^\times_{I}\otimes \Delta^\times_{J},
\end{eqnarray}
for all $I\in\cali{P}_{r,s}$, $J\in  \cali{P}_{s,t}$, and  $0<r<s<t$. We claim that the diagram 
\begin{eqnarray*}\label{pisubF}
\xymatrix{
\m{A}_{u,v}^\sharp  \ar[r]^-{\iota_{u,s,v}} & C^*_{\{p_{s,t}\}}(\cali{A})\otimes  C^*_{\{p_{s,t}\}}(\cali{A}) \\
\m{A}_{r,t}^\sharp  \ar[ru]_-{\iota_{r,s,t}}\ar[u]^-{\Delta^\sharp_{(r,t),(u,v)}} &}
\end{eqnarray*}
commutes for all  $0<u\leq r<s<t\leq v$, that is,  $\iota_{u,s,v}\Delta^\sharp_{(r,t),(u,v)}=\iota_{r,s,t}$. Indeed, let $I\in\cali{P}_{r,s}$ and $J\in \cali{P}_{s,t}$, where $0<r<s<t$. Consider the partions $uI:=\{u\}\cup I\in\cali{P}_{u,s}$ and $Jv:=J\cup\{v\}\in\cali{P}_{s,v}$. Using Proposition \ref{March23}, we have
\begin{eqnarray*}
\iota_{u,s,v}\Delta^\sharp_{(r,t),(u,v)}\Delta^\sharp_{I\cup J}&=&\iota_{u,s,v}\Delta^\sharp_{uI\cup Jv}\Delta^\times_{I\cup J,uI\cup Jv}= (\Delta^\times_{uI}\otimes \Delta^\times_{Jv})\Delta^\times_{I\cup J,uI\cup Jv}\\
&=&(\Delta^\times_{uI}\otimes \Delta^\times_{Jv})(\Delta^\times_{I,uI}\otimes \Delta^\times_{J,Jv})= \Delta^\times_{I}\otimes \Delta^\times_{J}\\&=&
\iota_{r,s,t}\Delta^\sharp_{I\cup J}.
\end{eqnarray*}
Because the set $\bigcup _{(I,J)\in \m{P}_{r,s}\times \m{P}_{s,t}}\Delta^\sharp_{I\cup J}(\m{A}_I\otimes \m{A}_J)$ is everywhere dense in $\m{A}_{r,t}^\sharp$, the commutativity of the diagram follows.

Using the fact the set of all open intervals $(u,v)$ containing the given interval $(r,t)$ is a cofinal subset of the partially ordered set $(I_{(0,\infty)},\subseteq)$, we have $$
\m{A}^{\diamond}_{\{p_{s,t}\}}=\limind
\left\{\m{A}_{u,v}^\sharp\,|\,(r,t)\subseteq (u,v)\in I_{(0,\infty)}\right\}.$$ As a result, we conclude, using the universal property of the inductive limit once again, that there is a unique *-monomorphism 
 $\iota_s:\m{A}^{\diamond}_{\{p_{s,t}\}}\to  C^*_{\{p_{s,t}\}}(\cali{A})\otimes C^*_{\{p_{s,t}\}}(\cali{A})$ that makes the diagram 
 \begin{eqnarray}\label{Nov10b}
\xymatrix{
\m{A}^{\diamond}_{\{p_{s,t}\}}  \ar[r]^-{\iota_s} & C^*_{\{p_{s,t}\}}(\cali{A})\otimes  C^*_{\{p_{s,t}\}}(\cali{A}) \\
\m{A}_{r,t}^\sharp  \ar[ru]_-{\iota_{r,s,t}}\ar[u]^-{\Delta^{\diamond}_{r,t}} &}
\end{eqnarray}
commutative for all  $0< r<s<t$, that is,  $\iota_s\Delta^{\diamond}_{r,t}=\iota_{r,s,t}$. 

Next, we compose the  *-monomorphism $\iota_s$ with the *-isomorphism $\delta$ of Theorem \ref{March24}, thus obtaining a *-monorphism $\Delta_s:C^*_{\{p_{s,t}\}}(\cali{A})\to C^*_{\{p_{s,t}\}}(\cali{A})\otimes C^*_{\{p_{s,t}\}}(\cali{A})$,
$$\Delta_s=\iota_s\delta,$$ for all $s>0$. 
Using and (\ref{Nov10b}) and (\ref{Nov10a}), one can immediately see that $\Delta_s$ satisfies the indentity \begin{eqnarray}\label{Nov10c}\Delta_s\Delta^\times_{I\cup J}=\Delta^\times_I\otimes \Delta^\times_J,\end{eqnarray} for all $I\in\cali{P}_{r,s}$, $J\in\cali{P}_{s,t}$, and $0<r<s<t$. Indeed, $\Delta_s\Delta^\times_{I\cup J}=\iota_s\delta\Delta^\times_{I\cup J}=\iota_s\Delta^{\diamond}_{r,t}\Delta^\sharp_{I\cup J}=\iota_{r,s,t}\Delta^\sharp_{I\cup J}= \Delta^\times_{I}\otimes \Delta^\times_{J},$ as needed.

It remains to be shown that the family $\{\Delta_s\}_{s>0}$ thus obtained satisfies (\ref{exx}). For this purpose, let  $0<r<s$ and $I\in\cali{P}$ be an arbitrary partition that contains both $r$ and $s$. Suppose that  $I\in\cali{P}_{p,t}$ for some  real numbers  $0<p<r<s<t$, and let $J\in\cali{P}_{p,r}$, $K\in\cali{P}_{r,s}$, $L\in\cali{P}_{s,t}$ be such that $I=J\cup K\cup L$. Using (\ref{Nov10c}) repeatedly, we have\begin{eqnarray*}
\left(\Delta_r\otimes\op{id}\right)\Delta_s\Delta_{I}^\times&=&\left(\Delta_r\otimes\op{id}\right) \Delta^\times_{J\cup K}\otimes \Delta^\times_{L}
=\Delta_r\Delta^\times_{J\cup K}\otimes \Delta^\times_{L}\\&=&\Delta^\times_{J}\otimes \Delta^\times_K\otimes \Delta^\times_{L}.
\end{eqnarray*}
Similarly, $(\op{id}\otimes\Delta_s)\Delta_r\Delta_{I}^\times=\Delta^\times_{J}\otimes \Delta^\times_K\otimes \Delta^\times_{L}$. Because the set $\{\Delta_{I}^\times (x)\,|\,x\in \m{A}_I,\,  I\in\cali{P},\,s,\,t\in I\}$ is everywhere dense in $C^*_{\{p_{s,t}\}}(\cali{A})$, the identity is satisfied.
\end{proof}
\begin{definition}
The family $\{\Delta_s\}_{s>0}$ of *-monomorphisms constructed above will be referred to as the one-parameter co-multiplication of the $C^*$-algebra $C^*_{\{p_{s,t}\}}(\cali{A})$.
\end{definition}

\begin{observation}\label{avner} (a) Let $p=\Delta^\times_{s,t}(p_{p,s})$. Then the projection $p$ is independent of the choice of the pair $0<s<t$, and $\Delta_s(p)=p\otimes p$, for all $s>0$.\\
(b) Consider two $C^*$-subproduct systems $\cali{A}=\left(\{\m{A}_{s,t}\}_{0<s<t}, \{\Delta_{r,s,t}\}_{0<r<s<t}\right)$ and $\cali{B}=\left(\{\m{B}_{s,t}\}_{0<s<t}, \{\Gamma_{r,s,t}\}_{0<r<s<t}\right)$, with units  
$\{p_{s,t}\}_{0<s<t}$ and $\{q_{s,t}\}_{0<s<t}$ respectively. If $\{\theta_{s,t}\}_{0<s<t}$ is a  monomorphism from $\cali{A}$ to $\cali{B}$ satisfying $\theta_{s,t}(p_{s,t})=q_{s,t}$, for all $0<s<t$, then the induced *-monomorphism of $C^*$-algebras $\theta:C^*_{\{p_{s,t}\}}(\cali{A})\to C^*_{\{q_{s,t}\}}(\cali{B})$, given by Proposition \ref{categorical theorem}, intertwines the one-parameter co-multiplications $\{\Delta_s\}_{s>0}$ and $\{\Gamma_s\}_{s>0}$ in the sense that $(\theta\otimes\theta)\Delta_s=\Gamma_s\theta$, for all $s>0$. \end{observation}

\begin{example}\label{June08} We identify and describe in concrete terms the $C^*$-algebra $C^*_1(\cali{A}_{\op{com}})$ of a unital $C^*$-subproduct system of commutative 
unital $C^*$-algebras $\cali{A}_{\op{com}}$ (see  Example \ref{examp1}) with respect to the trivial unit $p_{s,t}=1_{C(X_{s,t})}$, $0<s<t$. Let $\left(\{X_{s,t}\}_{0<s<t}, \{\chi_{r,s,t}\}_{0<r<s<t} \right)$ be a two-parameter multiplicative system of compact Hausdorff spaces so that the functions $\chi_{r,s,t}$ are all surjective. For any two partitions $I,\,J\in\cali{P}$, $I\subseteq J$, let $\chi_{I,J}^\times: X_J\to X_I$ be the continuous surjection that is defined with respect to the decomposition $J=\underline{I}\cup I^{\times}\cup \overline{I}$ as
\begin{eqnarray*}\label{vineriseara}\chi_{I,J}^\times =\left\{\begin{array}{ll}\chi_{I,J}, &\mbox{if}\;I,J\in \cali{P}_{s,t},\;s,t>0\\
(\chi_{I,I^{\times}})\comp(\pi_{J, I^\times}),
 &\mbox{otherwise}\end{array}\right.
\end{eqnarray*} where $\chi_{I,J}$ is as in Example \ref{examp11}, and $\pi_{J, I^\times}:X_J\to X_{I^\times}$ is the projection of $X_J$ onto $ X_{I^\times}.$ The resulting system $
\Bigl\{(X_I,\chi_{I,J}^\times)\,|\,I,\,J \in \cali{P},\; I\subseteq J\Bigr\}$
is a projective system of compact Hausdorff spaces over the partially ordered set $(\cali{P}, \subseteq)$, and its projective limit $$X=\limproj \Bigl\{(X_I,\chi_{I,J}^\times)\,|\,I,\,J \in \cali{P},\; I\subseteq J\Bigr\}$$ is a non-empty compact Hausdorff space. Moreover, the commutative $C^*$-algebra $C(X)$ is *-isomorphic to the $C^*$-algebra $C^*_1(\cali{A}_{\op{com}})$ of the unital $C^*$-subproduct system $\cali{A}_{\op{com}}=\left(\{C(X_{s,t})\}_{0<s<t}, \{\Delta_{r,s,t}\}_{0<r<s<t}\right)$.

 For any real numbers $0<u\leq s<t\leq v$, we also consider the continuous surjection $\chi^\sharp _{(s,t),(u,v)}: X_{u,v}^\sharp\to X_{s,t}^\sharp $, defined as $$\chi^\sharp _{(s,t),(u,v)}=\pi_{s,t}\comp \left(\op{id}_{X_{u,s}^\sharp}\otimes \chi^\sharp_{s,t,v}\right)^{-1}\comp (\chi^\sharp_{u,s,v})^{-1},\;\mbox{if}\; u< s<t< v,$$ and similar to the definition of $\Delta^\sharp _{(s,t),(u,v)}$ given in Proposition \ref{Nov06} in all other cases. Here $\pi_{s,t}:X_{u,s}^\sharp\times X_{s,t}^\sharp \times X_{t,v}^\sharp \to X_{s,t}^\sharp$ is the projection onto the space $ X_{s,t}^\sharp$. The system $
\left\{ \left(X_{s,t}^\sharp, \chi^\sharp _{(s,t),(u,v)}\right)\,|\, (s,t)\in I_{(0,\infty)}\right\}$
is also a projective system of compact Hausdorff spaces over the partially ordered set $(I_{(0,\infty)},\subseteq)$, and its projective limit 
$$X^{\diamond}=\limproj
\left\{ \left(X_{s,t}^\sharp, \chi^\sharp _{(s,t),(u,v)}\right)\,|\, (s,t)\in I_{(0,\infty)}\right\}$$  is a non-empty compact Hausdorff space that is homeomorphic to the space $X$. More precisely, as in Theorem \ref{March24}, there exists a unique homeomorphism $\psi:X^{\diamond}\to X$ so that the diagram
\begin{eqnarray*}\label{carerra}
\xymatrix{
X^\diamond  \ar[r]^-{\psi}\ar[d]_-{\chi_{s,t}^\diamond} & X\ar[d]^-{\chi^\times_{I}} \\
 X^\sharp_{s,t} \ar[r]_-{\chi^\sharp_I} & X_I
 }
\end{eqnarray*}
commutes, for all $I\in \cali{P}_{s,t}$ and $0<s<t$, where $\chi_I^\times$, $\chi^\sharp_I$ and $\chi^{\diamond}_{s,t}$ are the continuous surjections given by the projective limit construction.

Analogous to the construction of the *-monomorphisms  $\iota_{r,s,t}:\m{A}^\sharp_{r,t}\to C^*_{\{p_{s,t}\}}(\cali{A})\otimes  C^*_{\{p_{s,t}\}}(\cali{A})$ and $\iota_s:\m{A}^{\diamond}_{\{p_{s,t}\}}\to  C^*_{\{p_{s,t}\}}(\cali{A})\otimes C^*_{\{p_{s,t}\}}(\cali{A})$ in Theorem \ref{April2i}, one can construct two continuous surjections $\rho_{r,s,t}:X\times X\to X^\sharp_{r,t}$ and $\rho_s:X\times X\to X^\diamond$, uniquely determined by the conditions $\chi_{I\cup J}^\sharp\comp \rho_{r,s,t}=(\chi_I^\times) \times(\chi_J^\times)$, for 
all $I\in\cali{P}_{r,s}$, $J\in  \cali{P}_{s,t}$, and  $0<r<s<t$, respectively $\chi^{\diamond}_{r,t}\comp \rho_s=\rho_{r,s,t}$, for all  $0< r<s<t$. We then set \begin{eqnarray}\label{haiacasa}\chi_s=\psi\comp\rho_s:X\times X\to X,\;\;s>0,\end{eqnarray} and notice that $\chi^\times_{I\cup J}\comp\chi_s=(\chi^\times_I)\times (\chi^\times_J),$ for all $I\in\cali{P}_{r,s}$, $J\in\cali{P}_{s,t}$, and $0<r<s<t$, similar to (\ref{Nov10c}). By defining $\Delta_{\chi_s}(f)=f\comp \chi_s$, for all $f\in C(X)$ and $s>0$, we obtain that $\{\Delta_{\chi_s}\}_{s>0}$ is the one-parameter comultiplication on $C(X)\cong C^*_1(\cali{A}_{\op{com}})$, given by Theorem \ref{April2i}.
\end{example}


\section{Co-unital $C^*$-subproduct systems}\label{sec5.1}

Given a tensorial $C^*$- system  $\cali{A}=\left(\{\m{A}_{s,t}\}_{0<s<t}, \{\Delta_{r,s,t}\}_{0<r<s<t}\right)$, one can consider its conjugate system $\cali{A}^*=\left(\{\m{A}_{s,t}^*\}_{0<s<t}, \{\Delta_{r,s,t}^*\}_{0<r<s<t}\right)$, where $\m{A}_{s,t}^*$ is the conjugate space of the $C^*$-algebra $\m{A}_{s,t}$, for all $0<s<t$, and $\Delta_{r,s,t}^*:\m{A}_{r,s}^*\ba{\odot} \m{A}_{s,t}^*\to\m{A}_{r,t}^*$ is the restriction of the adjoint $\Delta_{r,s,t}^*$ of $\Delta_{r,s,t}$ to the completion $\m{A}_{r,s}^*\ba{\odot} \m{A}_{s,t}^*$ of the algebraic tensor product $\m{A}_{r,s}^*\odot \m{A}_{s,t}^*$ under the adjoint norm of the injective $C^*$-norm of $\m{A}_{r,s}\otimes \m{A}_{s,t}$ (see  \cite[Prop IV.4.10]{Tak}). The family of operators $\{\Delta_{r,s,t}^*\}_{0<r<s<t}$ thus defined satisfies the associativity law
 $\Delta_{r,s,u}^*\left(\op{id}_{\m{A}_{r,s}^*}\otimes \Delta_{s,t,u}^*\right)=
\Delta_{r,t,u}^*\left( \Delta_{r,s,t}^*\otimes \op{id}_{\m{A}_{t,u}^*}  \right)$, for all positive real numbers $0<r < s < t < u$, making the conjugate system $\cali{A}^*=\left(\{\m{A}_{s,t}^*\}_{0<s<t}, \{\Delta_{r,s,t}^*\}_{0<r<s<t}\right)$ into a ``two-parameter multiplicative tensorial system of Banach spaces''.

In this last section of our article, we focus primarily on families of bounded linear functionals that are invariant with respect to the multiplication of the conjugate system. Such families can be regarded as ``units'' of the conjugate system $\cali{A}^*$, or as ``co-units'' of $\cali{A}$.

\begin{definition}\label{labas}
Let $\cali{A}=\left(\{\m{A}_{s,t}\}_{0<s<t}, \{\Delta_{r,s,t}\}_{0<r<s<t}\right)$ be a tensorial $C^*$-system. A family $\{\varphi_{s,t}\}_{0<s<t}$ of bounded linear functionals $\varphi_{s,t}\in\m{A}_{s,t}^*$ is said to be co-multiplicative if they satisfy the co-multiplication law $$\varphi_{r,t}=\left(\varphi_{r,s}\otimes\varphi_{s,t} \right)\comp \Delta_{r,s,t},$$
 for all real numbers $0<r<s<t$. If, in addition, each $\varphi_{s,t}$ is a state of $\m{A}_{s,t}$, then the family $\{\varphi_{s,t}\}_{0<s<t}$ will be called a co-unit of $\cali{A}$. Any tensorial  $C^*$-system that admits a co-unit will be called co-unital. \end{definition}
 
The concept of co-multiplicative families of bounded linear functionals augments the notion of idempotent state of a compact quantum semigroup \cite{FS}, as indicated below. 
\begin{example}\label{zacus}An idempotent functional of a $C^*$-bialgebra $(\m{A}, \Delta)$ is a bounded linear functional $\varphi\in\m{A}^*$ that satisfies the co-multiplication law \begin{eqnarray}\label{hocapoca}
(\varphi\otimes \varphi)\comp \Delta=\varphi.\end{eqnarray} If $\varphi$ is as such, then the trivial family $\{\varphi_{s,t}\}_{0<s<t}$, where $\varphi_{s,t}=\varphi$, is a co-multiplicative family of bounded linear functionals of the trivial tensorial $C^*$-system $\cali{A}_{\op{triv}}$ associated with  $(\m{A}, \Delta)$, as introduced in Example \ref{quantgr}.
\end{example}
\begin{example}
Let $\cali{A}_{\op{fin}}=\left(\{\cali{B}(H_{s,t}\}_{0<s<t}, \{\op{ad}(U_{r,s,t})\}_{0<r<s<t}\right)$ be the $C^*$-subproduct system constructed from a Tsirelson subproduct system of finite dimensional Hilbert spaces $\cali{H}=\left(\{H_{s,t}\}_{0<s<t}, \{U_{r,s,t}\}_{0<r<s<t}\right)$ as in Example \ref{exa1}. For any two numbers $0<s<t$, let $\varphi_{s,t}=\op{tr}$, the trace on the matrix algebra  $\cali{B}(H_{s,t})$. Then the family $\{\varphi_{s,t}\}_{0<s<t}$ is co-multiplicative.
\end{example}
The co-multiplicative families of bounded functionals of a tensorial $C^*$-system of commutative $C^*$-algebras can be described directly in terms of their associated families of measures, as indicated in the following example.

 \begin{example} Let $\cali{A}_{\op{com}}=\left(\{C_0(X_{s,t})\}_{0<s<t}, \{\Delta_{r,s,t}\}_{0<r<s<t}\right)$ be the tensorial $C^*$-system constructed from a two-parameter multiplicative system of locally compact Hausdorff spaces  $\left(\{X_{s,t}\}_{0<s<t}, \{\chi_{r,s,t}\}_{0<r<s<t} \right)$, as in Example \ref{examp1}. The Riesz representation theorem establishes a direct correspondence between the class of co-multiplicative families $\{\varphi_{s,t}\}_{0<s<t}$ of bounded linear functionals of $\cali{A}_{\op{com}}$ and the class of families  $\{\mu_{s,t}\}_{0<s<t}$ of complex regular Borel measures $\mu_{s,t}$ on $X_{s,t}$ that satisfy the multiplication law \begin{eqnarray}\label{afinish}
\mu_{r,t}=(\chi_{r,s,t})_*(\mu_{r,s}\times \mu_{s,t}),\end{eqnarray}  for all real numbers $0<r<s<t$.

A particular case worth mentioning is that of a convolution semigroup of probability measures  $\{\mu_t\}_{t\geq 0}$ of $\mathbb{R}$ or, more generally, of an arbitrary locally compact group: any such convolution semigroup gives rise to the co-multiplicative family $\{\varphi_{s,t}\}_{0<s<t}$ of bounded linear functionals $\varphi_{s,t}(f)=\int f\,d\mu_{t-s}$, $f\in C_0(\mathbb{R})$, of the trivial $C^*$-subproduct system $\left(C_0(\mathbb{R}), \Delta\right)$, where $(\Delta f) (x,y)=f(x+y)$, for all $f\in C_0(\mathbb{R})$, and $x,\,y\in \mathbb{R}$.
\end{example}

In what follows, we will primarily focus on co-unital $C^*$-subproduct systems. To begin with, we show that the inductive dilation of a co-unital $C^*$-subproduct system is also co-unital.
\begin{proposition}\label{DK}
Let $\cali{A}=\left(\{\m{A}_{s,t}\}_{0<s<t}, \{\Delta_{r,s,t}\}_{0<r<s<t}\right)$ be a co-unital $C^*$-subproduct system and $\cali{A}^\sharp=\left(\{\m{A}_{s,t}^\sharp\}_{0<s<t},\,\{\Delta_{r,s,t}^\sharp\}_{0<r<s<t}\right)$ be its inductive dilation. Then for every co-unit $\{\varphi_{s,t}\}_{0<s<t}$ of $\cali{A}$, there exists a unique co-unit $\{\varphi_{s,t}^\sharp\}_{0<s<t}$ of  $\cali{A}^\sharp$, called the co-unit dilation of $\{\varphi_{s,t}\}_{0<s<t}$, such that
\begin{eqnarray}\label{Decaa}
\varphi^\sharp_{s,t}\comp \Delta^\sharp_{s,t}=\varphi_{s,t},
\end{eqnarray}
for all $0<s<t$, where $\{\Delta^\sharp_{s,t}\}_{0<s<t}$ is the inductive embedding of $\cali{A}$ into $\cali{A}^\sharp$.
\end{proposition}
\begin{proof}
Let $0<s < t$ be two fixed real numbers. For any finite partition $I\in\cali{P}_{s,t}$,  $ I=\{s=\iota_0<\iota_1<\iota_2<\,\dots<\iota_m<\iota_{m+1}=t\}$, consider the product state $\varphi_I=\varphi_{\iota_0,\iota_1}\otimes \varphi_{\iota_1,\iota_2}\otimes \cdots \otimes \varphi_{\iota_m,\iota_{m+1}}$ of the $C^*$-algebra  $\m{A}_I$. We claim that the resulting system of product states $\{\varphi_I\}_{I\in\cali{P}_{s,t}}$ leaves the system of *-monomorphism $\{\Delta_{I,J}\}_{I\subset J\in\cali{P}_{s,t}}$ invariant, in the sense that 
\begin{eqnarray}\label{june14}\varphi_J\comp\Delta_{I,J}=\varphi_I, \end{eqnarray} 
for all $I,\, J\in\cali{P}_{s,t}$, $I\subseteq J$. Indeed, by writing $J=I_0\cup I_1\cup \cdots \cup I_m$ and $\Delta_{I,J}=\Delta_{\{\iota_0,\iota_1\},I_0}\otimes \cdots \otimes \Delta_{\{\iota_m,\iota_{m+1}\},I_m}$, we obtain that $$\varphi_J\comp \Delta_{I,J}=\varphi_{I_0}\comp \Delta_{\{\iota_0,\iota_1\},I_0}\otimes \cdots \otimes \varphi_{I_m}\comp \Delta_{\{\iota_m,\iota_{m+1}\},I_m}.$$ As a result, we can assume without losing generality that the partition $I$ is a trivial. Let then $I=\{s,t\}$ and $J=\{s=j_0<j_1<\iota_2<\,\dots<j_n<j_{n+1}=t\}$.  Expanding $\Delta_{I,J}$ as in Remark \ref{Nov14a}, we obtain that 
\begin{align*}\varphi_J\comp \Delta_{I,J}&=(\varphi_{j_0,j_1}\otimes \cdots \otimes \varphi_{j_n,j_{n+1}})\comp (\Delta_{j_0,j_1,j_2}\otimes \op{id}_{j_2,j_3}\otimes\cdots )\cdots \Delta_{j_0,j_n,j_{n+1}}
\\&=[(\varphi_{j_0,j_1}\otimes \varphi_{j_1,j_2})\comp \Delta_{j_0,j_1,j_2}\otimes \varphi_{J\setminus \{j_0,j_1\}}] \Delta_{\{j_0,j_{n+1}\},J\setminus \{j_1\}}
\\&=\left(\varphi_{j_0,j_2}\otimes \varphi_{J\setminus \{j_0,j_1\}}\right)\comp \Delta_{\{j_0,j_{n+1}\},J\setminus \{j_1\}}\\&=\varphi_{J\setminus \{j_1\}}\comp \Delta_{\{j_0,j_{n+1}\},J\setminus \{j_1\}}. 
\\&=\cdots\cdots\cdots\cdots\\
&=\varphi_{J\setminus \{j_1,\ldots j_{n-1}\}}\comp \Delta_{\{j_0,j_{n+1}\},J\setminus \{j_1,\ldots j_{n-1}\}}
\\&=(\varphi_{j_0,j_n}\otimes \varphi_{j_n,j_{n+1}})\comp \Delta_{j_0,j_n,j_{n+1}}=\varphi_{j_0,j_{n+1}}=\varphi_I,
\end{align*} 
as claimed. Consequently, one can consider the inductive limit $$\varphi^\sharp_{s,t}:=\limind _{I\in\cali{P}_{s,t}}\varphi_I$$ of the family $\{\varphi_I\}_{I\in\cali{P}_{s,t}}$, i.e.,
the unique state $\varphi^\sharp_{s,t}$ of the $C^*$-algebra $\m{A}^\sharp_{s,t}$ that satisfies the compatibility condition \begin{eqnarray}\label{hass}
\varphi_{s,t}^\sharp\comp \Delta_I^\sharp=\varphi_I,\end{eqnarray}
for all $I\in \cali{P}_{s,t}$. 

We claim that the resulting family $\{\varphi^\sharp_{s,t}\}_{0<s<t}$ of states of the $C^*$-product system $\cali{A}^\sharp$ is co-multiplicative. Indeed, for any positive real numbers $0<r<s<t$ and any finite partitions $I\in  \cali{P}_{r,s}$, $J\in \cali{P}_{s,t}$, one can use (\ref{compas}) to obtain that \begin{eqnarray*}
 \left(\varphi_{r,s}^\sharp\otimes\varphi_{s,t}^\sharp \right)\comp \Delta_{r,s,t}^\sharp\Delta_{I\cup J}^\sharp&=& \left(\varphi_{r,s}^\sharp\otimes\varphi_{s,t}^\sharp \right)\comp\left(\Delta_I^\sharp\otimes \Delta_J^\sharp\right)=\varphi_I\otimes \varphi_J=
 \varphi_{I\cup J}\\&=&\varphi_{r,t}^\sharp\comp\Delta_{I\cup J}^\sharp.\end{eqnarray*} Because the set $\bigcup _{(I,J)\in \m{P}_{r,s}\times \m{P}_{s,t}}\Delta^\sharp_{I\cup J}(\m{A}_{I\cup J})$ is everywhere dense in $\m{A}_{r,t}^\sharp$, we infer that the family $\{\varphi^\sharp_{s,t}\}_{0<s<t}$ is co-multiplicative. The uniqueness of $\{\varphi^\sharp_{s,t}\}_{0<s<t}$ in relation to (\ref{Decaa}) can be easily shown in a manner similar to that used in the proof of Proposition \ref{categorical theorem0}.
\end{proof}

We notice that if $\{\varphi_{s,t}\}_{0<s<t}$ and $\{\psi_{s,t}\}_{0<s<t}$ are two equivalent co-units of a $C^*$-subproduct system $\cali{A}$, in the sense that there is an automorphism $\{\theta_{s,t}\}_{0<s<t}$ of $\cali{A}$ so that $\varphi_{s,t}=\psi_{s,t}\comp\theta_{s,t}$, for all $0<s<t$, then their dilations $\{\varphi_{s,t}^\sharp\}_{0<s<t}$ and $\{\psi_{s,t}^\sharp\}_{0<s<t}$ are also equivalent. Moreover $\varphi_{s,t}^\sharp=\psi_{s,t}^\sharp\comp\theta_{s,t}^\sharp,$ for all $0<s<t$, where $\{\theta_{s,t}^\sharp\}_{0<s<t}$ is the isomorphism of $\cali{A}^\sharp$ induced by  $\{\theta_{s,t}\}_{0<s<t}$, as in Proposition \ref{categorical theorem0}.
\begin{example}\label{jun17}
Let $\left(\{X_{s,t}\}_{0<s<t}, \{\chi_{r,s,t}\}_{0<r<s<t} \right)$ be a two-parameter multiplicative system of compact Hausdorff spaces so that the functions $\chi_{r,s,t}$ are all surjective. Suppose that $\{\mu_{s,t}\}_{0<s<t}$ is a family of Borel probability measures $\mu_{s,t}$ on $X_{s,t}$, $0<s<t$, which satisfies the multiplication law (\ref{afinish}). For any two positive real numbers $0<s<t$ and any partition $I=\{s=\iota_0<\iota_1<\iota_2<\,\dots<\iota_m<\iota_{m+1}=t\}\in \cali{P}_{s,t}$, consider the product measure $\mu_I=
\mu_{\iota_0, \iota_1}\times \mu_{\iota_1,\iota_2}\times \dots\times \mu_{\iota_m,\iota_{m+1}}$ on $X_I$. If $J\in \cali{P}_{s,t}$ is a  refinement of $I$, then we deduce from (\ref{afinish}) that the continuous function $\chi_{I,J}$, defined by (\ref{June16}), is measure preserving, in the sense that $\mu_I=(\chi_{I,J})_*\mu_J$. Consequently, the system $
\{(X_I, \op{Bor}(X_I), \mu_I)\}_{I\in \cali{P}_{s,t} }
$ is a projective system of probability spaces over the partially ordered set $(\cali{P}_{s,t},  \subseteq)$, with connecting mappings $\chi_{I, J}$.

Next, we consider the set algebra $\m{M}_{s,t}=\bigcup_{I\in\cali{P}_{s,t}}(\chi_I^\sharp)^{-1}(\op{Bor}(X_I))$ on $X^\sharp_{s,t}$, and the finitely additive set function $m_{s,t}$ on $\m{M}_{s,t}$, acting as $$m_{s,t}\left((\chi_I^\sharp)^{-1}(B)\right)=\mu_I(B),$$ for every $I\in\cali{P}_{s,t}$ and every Borel subset $B$ of $X_I$. Because the spaces $X_I$ are all compact Hasudorff spaces and the Borel $\sigma$-algebra $\op{Bor}(X_{s,t}^\sharp)$ coincides with the $\sigma$-algebra generated by  $\m{M}_{s,t}$, we deduce, using a classical generalization of Kolmogorov's extension theorem (see e.g. \cite{Chol}), that the set function $m_{s,t}$ can be extended to a Borel probability measure $\mu_{s,t}^\sharp$ on $X_{s,t}^\sharp$, uniquely determined by the condition $(\chi_I^\sharp)_*\mu_{s,t}^\sharp=\mu_I$, for all $I\in\cali{P}_{s,t}$ and $0<s<t$.

It can be easily verified that the family of Borel probability measures  $\{\mu_{s,t}^\sharp\}_{0<s<t}$ thus obtained satisfies the multiplication law (\ref{afinish}) in relation to the family of homeomorphisms $\{\chi_{r,s,t}^\sharp\}_{0<r<s<t}$ introduced in Example \ref{examp11}. 
The co-multiplicative family of states of the $C^*$-product system $\cali{A}_{\op{com}}^\sharp=\left(\{C(X_{s,t}^\sharp)\}_{0<s<t},\,\{\Delta_{r,s,t}^\sharp\}_{0<r<s<t}\right)$ associated with this family of measures by means of the Riesz representation corresponds to the dilation of the co-unit of the $C^*$-subproduct system $\cali{A}_{\op{com}}=\left(\{C(X_{s,t})\}_{0<s<t}, \{\Delta_{r,s,t}\}_{0<r<s<t}\right)$ that is associated with the family of measures  $\{\mu_{s,t}\}_{0<s<t}$.
\end{example}
The notion of idempotent state of a $C^*$-algebra  $(\m{A}, \Delta)$, mentioned in Example \ref{zacus}, can be adapted to the setting of the $C^*$-algebra $C^*_{\{p_{s,t}\}}(\cali{A})$ of a unital $C^*$-subproduct system $\cali{A}=\left(\{\m{A}_{s,t}\}_{0<s<t}, \{\Delta_{r,s,t}\}_{0<r<s<t}\right)$ with unit $\{p_{s,t}\}_{0<s<t}$, the role of the co-multiplication $\Delta$ being played by the one-parameter co-multiplication $\{\Delta_s\}_{s>0}$, given by Theorem \ref{April2i}.
\begin{definition}
A state $\varphi$ of the $C^*$-algebra $C^*_{\{p_{s,t}\}}(\cali{A})$ is said to be an idempotent state if $(\varphi\otimes \varphi)\comp \Delta_s=\varphi,$ for all $s>0$; we also require that $p\notin \m{N}_{\varphi}$, the left kernel of $\varphi$, where $p=\Delta^\times_{s,t}(p_{p,s})$ for some $0<s<t$.
\end{definition}

The following two results establish a direct correspondence between the co-units of a unital $C^*$-subproduct system and the idempotent states of the $C^*$-algebra of that system. We notice that when applied to a unital trivial $C^*$-subproduct system (see Example \ref{quantgr}) with trivial unit, these results recovers Sch\"{u}rmann's reconstruction of a white noise from its 1-dimensional distribution \cite[pg. 38-40]{Sch}
\begin{proposition}\label{DK2} Let $\cali{A}=\left(\{\m{A}_{s,t}\}_{0<s<t}, \{\Delta_{r,s,t}\}_{0<r<s<t}\right)$ be a unital $C^*$-subproduct system with unit $\{p_{s,t}\}_{0<s<t}$. If the $C^*$-algebra $C^*_{\{p_{s,t}\}}(\cali{A})$ admits an  idempotent state $\varphi$, then $\cali{A}$ is co-unital and the family $\{\varphi_{s,t}\}_{0<s<t}$ of marginal states \begin{eqnarray}\label{margina}\varphi_{s,t}=\varphi\comp \Delta_{s,t}^\times,\;\;\;0<s<t, \end{eqnarray} is a co-unit of $\cali{A}$.
\end{proposition}
\begin{proof} 
We deduce immediately from (\ref{Nov10c}) that $(\varphi\comp \Delta_I^\times)\otimes (\varphi\comp \Delta_J^\times)=\varphi\comp \Delta_{I\cup J}^\times$ for all partitions $I\in \cali{P}_{r,s}$ and $J\in \cali{P}_{s,t}$, where $0<r<s<t$. By taking $I$ and $J$ to be trivial paritions, $I=\{r,s\}$, respectively $J=\{s,t\}$, and using the fact that $\Delta_{r,s,t}=\Delta _{\{r,t\},\{r,s,t\}}^\times$, we obtain that $\left(\varphi_{r,s}\otimes\varphi_{s,t} \right)\comp \Delta_{r,s,t}=
\varphi\comp\Delta^\times_{\{r,s,t\}}\Delta _{\{r,t\},\{r,s,t\}}^\times=\varphi_{r,t}
$, as needed. Moreover, because $\varphi_{s,t}(p_{s,t})=1$, we obtain that $\|\varphi_{s,t}\|=1$, for all $0<s<t.$   \end{proof}

\begin{theorem}\label{DK1} Suppose that $\cali{A}=\left(\{\m{A}_{s,t}\}_{0<s<t}, \{\Delta_{r,s,t}\}_{0<r<s<t}\right)$ is a unital co-unital $C^*$-subproduct system with unit $\{p_{s,t}\}_{0<s<t}$ and co-unit $\{\varphi_{s,t}\}_{0<s<t}$ satisfying $\varphi_{s,t}(p_{s,t})=1$, for all $0<s<t$. Then there exists a unique idempotent state $\varphi$ of the $C^*$-algebra  $C^*_{\{p_{s,t}\}}(\cali{A})$ such that $\{\varphi_{s,t}\}_{0<s<t}$ are the  marginal states of $\varphi$.
\end{theorem}

\begin{proof} The idempotent state $\varphi$ will be constructed as the inductive limit of the family of product states $\{\varphi_I\}_{I\in\cali{P}}$, defined as in the proof of Proposition \ref{DK} 
For this purpose, it is enough to show that \begin{eqnarray}\label{june15}\varphi_J\comp\Delta_{I,J}^\times=\varphi_I, \end{eqnarray} 
for all $I,\,J\in\cali{P}$, $I\subseteq J$. To check this compatibility condition, we note that if $I,\,J\in\cali{P}_{s,t}$, for some $0<s<t$, then (\ref{june15}) is exactly (\ref{june14}). Moreover, if $I$ and $J$ are arbitrary finite partitions, then by decomposing $J$ with respect to $I$ as $J=\underline{I}\cup I^{\times}\cup \overline{I}$, we obtain that $$\varphi_J\comp \Delta_{I,J}^\times(x)=(\varphi_{\underline{I}}\otimes \varphi_{I^{\times}}\otimes \varphi_{\overline{I}})({p}_{\underline{I}}\otimes \Delta_{I,I^{\times}}(x)\otimes {p}_{\overline{I}})=\varphi_{I^{\times}}\comp\Delta_{I,I^{\times}}(x)=\varphi_I(x),$$ for all $x\in\m{A}_I$. 

Let $\varphi=\limind_{I\in\cali{P}}\varphi_I$ be the inductive limit of the inductive family of states $\{\varphi_I\}_{I\in\cali{P}}$. Therefore $\varphi$ the unique state of the $C^*$-algebra $C^*(\cali{A})$ that satisfies the compatibility condition $\varphi\comp \Delta^\times_I=\varphi_I$,
 for every partition $I\in \cali{P}$. In particular, $\varphi$ satisfies (\ref{margina}).
Moreover, using (\ref{Nov10c}), we obtain that for any number $s$, we have $$(\varphi\otimes \varphi)\comp\Delta_s\Delta_{I\cup J}^\times=(\varphi\comp \Delta_I^\times)\otimes (\varphi\comp \Delta_J^\times)=\varphi_{I\cup J}=\varphi\comp \Delta_{I\cup J}^\times,$$ for all $I\in\cali{P}_{r,s}$, $J\in\cali{P}_{s,t}$, and $0<r<s<t$. It then follows that $\varphi$ is an idempotent state. Uniqueness is immediate.
\end{proof}

\begin{example}Let $\left(\{X_{s,t}\}_{0<s<t}, \{\chi_{r,s,t}\}_{0<r<s<t} \right)$ and $\{\mu_{s,t}\}_{0<s<t}$ be as in Example \ref{jun17}. Then for any two partitions $I,\,J\in \cali{P}$, $I\subseteq J$, we have $$\mu_I=(\chi_{I,J}^\times)_*\mu_J,$$ where $\chi_{I,J}^\times$ is as in (\ref{vineriseara}). Therefore the system $
\{(X_I, \op{Bor}(X_I), \mu_I)\}_{I\in \cali{P}}
$ is a projective system of probability spaces over the partially ordered set $(\cali{P},  \subseteq)$, with connecting mappings $\chi_{I, J}^\times$. Arguing as in Example \ref{jun17}, one can construct a Borel probability measure $\mu$ on $X$ (see Example \ref{June08}) that is uniquely determined by the condition $(\chi_I^\times)_*\mu=\mu_I$, for every $I\in\cali{P}$. Moreover, $\mu$ is an idempotent measure in relation to the functions $\chi_s$ defined in (\ref{haiacasa}), in the sense that $\mu={(\chi_s)}_*(\mu\times \mu)$, for all $s>0$. This can be shown by employing the technique used in Section \ref{sec5+.1}. 

The state $\varphi$ of $C(X)$ induced by the measure $\mu$ is therefore an idempotent state that corresponds to the idempotent state of $C^*_1(\cali{A}_{\op{com}})$, given by Theorem \ref{DK1}, constructed from the co-unit $\{\varphi_{s,t}\}_{0<s<t}$ associated with $\{\mu_{s,t}\}_{0<s<t}$, with respect to the trivial unit $p_{s,t}=1_{C(X_{s,t})}$, $0<s<t$.
\end{example}

Co-units of $C^*$-subproduct systems give rise to a Tsirelson subproduct systems of Hilbert spaces, as shown below.
\begin{proposition} \label{harici}Let $\cali{A}=\left(\{\m{A}_{s,t}\}_{0<s<t}, \{\Delta_{r,s,t}\}_{0<r<s<t}\right)$ be a co-unital $C^*$-subproduct system and $\{\varphi_{s,t}\}_{0<s<t}$ be a co-unit of $\cali{A}$. Let $(\pi_{s,t}, H_{s,t})$ be the GNS representation of $\m{A}_{s,t}$ associated with $\varphi_{s,t}$, for every $0<s<t$. Then the family $\{H_{s,t}\}_{0<s<t}$ can be assembled into a Tsirelson subproduct system of Hilbert spaces.
\end{proposition}
\begin{proof} For any two real numbers $0<s<t$, let $\m{N}_{\varphi_{s,t}}$ be the left kernel of $\varphi_{s,t}$ and $\eta_{\varphi_{s,t}}(x)$ be the coset $x+\m{N}_{\varphi_{s,t}}$ of a representative $x\in \m{A}_{s,t}$ in the quotient space $\m{A}_{s,t}/\m{N}_{\varphi_{s,t}}$. Then the Hilbert space $H_{s,t}$ in the GNS construction associated with the state $\varphi_{s,t}$
is the completion of $\m{A}_{s,t}/\m{N}_{\varphi_{s,t}}$ with respect to the inner product $\ip{\eta_{\varphi_{s,t}}(x)}{\eta_{\varphi_{s,t}}(y)}=\varphi_{s,t}(y^*x)$, for all $x,\,y\in \m{A}_{s,t}$.

Using the co-multiplicativity of the family $\{\varphi_{s,t}\}_{0<s<t}$, we deduce that $\Delta_{r,s,t}(\m{N}_{\varphi_{r,t}})\subset\m{N}_{\varphi_{r,s}\otimes \varphi_{s,t}},$ for all $0<r<s<t$. Consequently, by identifying the Hilbert space $H_{r,s}\otimes H_{s,t}$ with the Hilbert space in the GNS construction associated with the product state $\varphi_{\{r,s,t\}}=\varphi_{r,s}\otimes \varphi_{s,t}$ via the unitary operator $$\eta_{\varphi_{r,s}}(x)\otimes \eta_{\varphi_{s,t}}(y)\mapsto \eta_{\varphi_{\{r,s,t\}}}(x\otimes y),\;\;x\in \m{A}_{r,s},\;y\in \m{A}_{s,t},$$ one can consider the operator $V_{r,s,t}:H_{r,t}\to H_{r,s}\otimes H_{t,s}$ that acts as \begin{eqnarray}\label{noparty}V_{r,s,t}(\eta_{\varphi_{r,t}}(x))=\eta_{\varphi_{\{r,s,t\}}}\left(\Delta_{r,s,t}(x) \right),\end{eqnarray} for all $x\in  \m{A}_{r,t}$. We notice that $V_{r,s,t}$ is an isometry. Indeed, for any  $x,\,y\in  \m{A}_{r,t}$, we have 
 \begin{eqnarray*}
 \ip{V_{r,s,t}(\eta_{\varphi_{r,t}}(x))}{V_{r,s,t}(\eta_{\varphi_{r,t}}(y))}&=&\ip{\eta_{\varphi_{\{r,s,t\}}}\left(\Delta_{r,s,t}(x)\right)}{\eta_{\varphi_{\{r,s,t\}}}\left(\Delta_{r,s,t}(y)\right)}\\&=&\varphi_{\{r,s,t\}}\left(\Delta_{r,s,t}(y)^*\Delta_{r,s,t}(x) \right)=\varphi_{r,t}(y^*x)\\&=&\ip{\eta_{\varphi_{r,t}}(x)}{\eta_{\varphi_{r,t}}(y)}.\end{eqnarray*}
It is also clear that the family $\{V_{r,s,t}\}_{0<r<s<t}$ satisfies the co-associativity law (\ref{Jan24cc}). Therefore
\begin{eqnarray}\label{alohaa}\cali{H}_{\{\varphi_{s,t}\}}=\left(\{H_{s,t}\}_{0<s<t}, \{V_{r,s,t}\}_{0<r<s<t}\right)\end{eqnarray} is a Tsirelson subproduct system of Hilbert spaces.
\end{proof}

\begin{definition}The system (\ref{alohaa}) will be called the Tsirelson subproduct system of Hilbert spaces associated with the co-unit $\{\varphi_{s,t}\}_{0<s<t}$. 
\end{definition}
\begin{observation}
 Suppose that $\cali{A}=\left(\{\m{A}_{s,t}\}_{0<s<t}, \{\Delta_{r,s,t}\}_{0<r<s<t}\right)$ is a unital co-unital $C^*$-subproduct system with unit $\{p_{s,t}\}_{0<s<t}$ and co-counit $\{\varphi_{s,t}\}_{0<s<t}$. If $\varphi_{s,t}(p_{s,t})=1$, for all $0<s<t$, then $\{\eta_{\varphi_{s,t}}(p_{s,t})\}_{0<s<t}$ is a normalized unit of the Tsirelson subproduct system of Hilbert spaces associated with the co-unit $\{\varphi_{s,t}\}_{0<s<t}$.
\end{observation}

The Tsirelson subproduct system of Hilbert spaces associated with a co-unit of a $C^*$-subproduct system of commutative $C^*$-algebras can be easily identified, as shown below.
\begin{example}
Let $\cali{A}_{\op{com}}=\left(\{C_0(X_{s,t})\}_{0<s<t}, \{\Delta_{r,s,t}\}_{0<r<s<t}\right)$ be the $C^*$-subproduct system constructed from a two-parameter multiplicative system of locally compact Hausdorff spaces $\left(\{X_{s,t}\}_{0<s<t}, \{\chi_{r,s,t}\}_{0<r<s<t} \right)$, where the  functions $\chi_{r,s,t}$ are assumed to be surjective. If $\{\varphi_{s,t}\}_{0<s<t}$ is a co-unit of $\cali{A}_{\op{com}}$, and $\{\mu_{s,t}\}_{0<s<t}$ is the associated family of Borel probability measures $\mu_{s,t}$ on $X_{s,t}$, then $$\cali{H}_{\{\varphi_{s,t}\}}=\left(\{L^2(X_{s,t}, \mu_{s,t})\}_{0<s<t}, \{V_{r,s,t}\}_{0<r<s<t}\right),$$ where 
$V_{r,s,t}f= f\comp\chi_{r,s,t},$
for all $f\in C_0(X_{r,t})$ and $0<r<s<t$.
\end{example}
\begin{remark}\label{monster} Similar to the Bhat-Mukherjee dilation of an Arveson subproduct system of Hilbert spaces \cite{Bhat-M}, any Tsirelson subproduct system of Hilbert spaces $\cali{H}=\left(\{H_{s,t}\}_{0<s<t}, \{V_{r,s,t}\}_{0<r<s<t}\right)$ can be dilated to a Tsirelson product system of Hilbert spaces $\cali{H}^\sharp=\left(\{H_{s,t}^\sharp\}_{0<s<t}, \{V_{r,s,t}^\sharp\}_{0<r<s<t}\right)$. The procedure for constructing $\cali{H}^\sharp$ is similar to that used in Section \ref{ch3.1} to construct the inductive dilation of a $C^*$-subproduct system.
We briefly describe the most important details of this construction here, leaving their completion to the discretion of the reader.

Let $\cali{H}=\left(\{H_{s,t}\}_{0<s<t}, \{V_{r,s,t}\}_{0<r<s<t}\right)$ be a Tsirelson subproduct system of Hilbert spaces.  For any two positive real numbers  $0<s < t$ and any partition $I\in \cali{P}_{s,t} $, $I=\{s=\iota_0<\iota_1<\iota_2<\,\dots<\iota_m<\iota_{m+1}=t\},$ consider the Hilbert space \begin{eqnarray}\label{Jan25}H_I=
H_{\iota_0, \iota_1}\otimes H_{\iota_1,\iota_2}\otimes \dots\otimes H_{\iota_m,\iota_{m+1}}.\end{eqnarray} and the isometric operator $V_{\{s,t\},I}:H_{s,t}\to H_I$, which is defined analogous to the *-monomorphism $\Delta_{\{s,t\},I}$ of Definition \ref{Jan21}, i.e.,
 \begin{eqnarray}V_{\{s,t\},I}=\left\{\begin{array}{llll}V_{\iota_0,\iota_1,\iota_2}, &\;m=1\\\left(V_{\{\iota_0,\iota_m\},I\setminus\{\iota_{m+1}\}}\otimes 1_{\iota_m,\iota_{m+1}}\right)V_{\iota_0,\iota_m,\iota_{m+1}},
&\; m\geq 2
\end{array}\right.\end{eqnarray} where $1_{\iota_m,\iota_{m+1}}$
is the identity operator on $H_{\iota_m,\iota_{m+1}}$. Moreover, if $J\in \cali{P}_{s,t}$ is an arbitrary refinement of $I$, decomposed as $J=I_0\cup \ldots \cup I_m$,  we also consider the isometry $V_{I, J}:H_I\to H_J$, \begin{eqnarray}\label{Jan25i}V_{I,J}=V_{\{\iota_0,\iota_1\}, I_0}\otimes V_{\{\iota_1,\iota_2\}, I_1}\otimes \cdots V_{\{\iota_m,\iota_{m+1}\}, I_m}.\end{eqnarray}
As in Proposition \ref{lemma inductive limit-c}, the system $\Bigl\{(H_I,V_{I,J})\,|\, I,\,J \in \cali{P}_{s,t},\; I\subseteq J \Bigr\}
$ is an inductive system of Hilbert spaces, for all $0<s<t$, and let $$H_{s,t}^\sharp=\limind\, \Bigl\{(H_I, V_{I,J})\,|\, I,\,J \in \cali{P}_{s,t},\; I\subseteq J\Bigr\}$$ be its inductive limit with associated connecting isometries $V_I^\sharp :H_I\to H^\sharp_{s,t}$, $I\in \cali{P}_{s,t}$. Similar to Theorem \ref{star-isomorphism theorem}, there exists a unitary operator $V_{r,s,t}^\sharp:H_{r,t}^\sharp\rightarrow H_{r,s}^\sharp\otimes H_{s,t}^\sharp$, uniquely determined by the condition  \begin{eqnarray}\label{holab}V_{r,s,t}^\sharp V_{I\cup J}^\sharp=V_I^\sharp\otimes V_J^\sharp,\end{eqnarray}
for all $I\in\cali{P}_{r,s}$, $J\in  \cali{P}_{s,t}$. The resulting family $\{V_{r,s,t}^\sharp\}_{0<r<s<t}$ satisfies the co-associativity law (\ref{Jan24cc}), thus making the system 
$$\cali{H}^\sharp=\left(\{H_{s,t}^\sharp\}_{0<s<t},\,\{V_{r,s,t}^\sharp\}_{0<r<s<t}\right)$$ a Tsirelson product system of Hilbert spaces. This system will be called the Bhat-Mukherjee dilation of the Tsirelson subproduct system $\cali{H}$.
\end{remark}

\begin{theorem} \label{harici1}
Let $\cali{A}=\left(\{\m{A}_{s,t}\}_{0<s<t}, \{\Delta_{r,s,t}\}_{0<r<s<t}\right)$ be a co-unital $C^*$-subproduct system and $\{\varphi_{s,t}\}_{0<s<t}$ be a co-unit of $\cali{A}$. 
Consider the Tsirelson subproduct system of Hilbert spaces $\cali{H}_{\{\varphi_{s,t}\}}=\left(\{H_{s,t}\}_{0<s<t}, \{V_{r,s,t}\}_{0<r<s<t}\right)$ associated with $\{\varphi_{s,t}\}_{0<s<t}$, and let $\cali{H}_{\{\varphi_{s,t}\}}^\sharp$ be its Bhat-Mukherjee dilation. If $\cali{H}_{\{\varphi_{s,t}^\sharp\}}=\left(\{H_{s,t}'\}_{0<s<t}, \{V_{r,s,t}'\}_{0<r<s<t}\right)$ is the Tsirelson product system of Hilbert spaces associated with the co-unit dilation $\{\varphi_{s,t}^\sharp\}_{0<s<t}$ of $\{\varphi_{s,t}\}_{0<s<t}$, then $\cali{H}_{\{\varphi_{s,t}^\sharp\}}$ and  $\cali{H}_{\{\varphi_{s,t}\}}^\sharp$ are isomorphic Tsirelson product systems of Hilbert spaces. 
\end{theorem}

\begin{proof}
Let $0<s<t$ be two fixed real numbers. For any partition $I\in\cali{P}_{s,t}$ of the form $I=\{s=\iota_0<\iota_1<\iota_2<\,\dots<\iota_m<\iota_{m+1}=t\},$ we identify, as in the proof of Proposition \ref{harici}, the Hilbert space $H_I$, defined in (\ref{Jan25}), with the Hilbert space
in the GNS construction of the product state $\varphi_I=\varphi_{\iota_0,\iota_1}\otimes \varphi_{\iota_1,\iota_2}\otimes \cdots \otimes \varphi_{\iota_m,\iota_{m+1}}$ of the $C^*$-algebra $\m{A}_I$. Keeping the notation used in the proof of Proposition \ref{harici}, we deduce from (\ref{hass}) that $\Delta_I^\sharp\left(\m{N}_{\varphi_I}\right)\subseteq \m{N}_{\varphi_{s,t}^\sharp}$. Consequently, one can define the operator $V_I:H_I\to H'_{s,t}$ by $$
V_I(\eta _{\varphi_I}(x))=\eta_{\varphi_{s,t}^\sharp}\left(\Delta_I^\sharp(x)\right),$$ for every $x\in \m{A}_I$. Using (\ref{hass}) and reasoning as in the proof of Proposition \ref{harici} again, we deduce that $V_I$ is an isometry. We also note that $V_JV_{I,J}=V_I$, for all $I,\, J\in \cali{P}_{s,t}$, $I\subseteq J$, where $V_{I,J}$ is the connecting isometric operator defined in (\ref{Jan25i}). It is enough to check this identity when $I=\{s,t\}$. If this is the case, then \begin{eqnarray*}
V_JV_{\{s,t\},J}(\eta_{\varphi_{s,t}}(x))&=&V_J(\eta_{\varphi_J}(\Delta_{\{s,t\}, J}(x)))=\eta_{\varphi_{s,t}^\sharp}(\Delta_J^\sharp\Delta_{\{s,t\}, J}(x))\\&=&\eta_{\varphi_{s,t}^\sharp}( \Delta^\sharp_{\{s,t\}}(x))=V_{\{s,t\}}(\eta_{\varphi_{s,t}}(x)),
\end{eqnarray*} for all $x\in\m{A}_{s,t}$, as required.

Additionally, we notice that the set $\bigcup _{I\in \m{P}_{s,t}}V_{I}(H_{I})$ is everywhere dense in $H'_{s,t}$ because the set $\bigcup _{I\in \m{P}_{s,t}}\Delta^\sharp_{I}(\m{A}_{I})$ is everywhere dense in $\m{A}_{s,t}^\sharp$. Consequently, there exists a unique unitary operator $Z_{s,t}:H^\sharp_{s,t}\to H'_{s,t}$ such that $Z_{s,t}V_I^\sharp=V_I$, for all $I\in\cali{P}_{s,t}$. 

We claim that the resulting family of unitary operators $\{Z_{s,t}\}_{0<s<t}$ is an isomorphism of Tsirelson product systems of Hilbert spaces, i.e., it satisfies $(Z_{r,s}\otimes Z_{s,t})V_{r,s,t}^\sharp=V'_{r,s,t}Z_{r,t}$ for all $0<r<s<t.$ Indeed, for any two partitions $I\in\cali{P}_{r,s}$, $J\in  \cali{P}_{s,t}$, we have
\begin{eqnarray*}
(Z_{r,s}\otimes Z_{s,t})V_{r,s,t}^\sharp V_{I\cup J}^\sharp&\stackrel{(\ref{holab})}{=}&Z_{r,s} V_{I}^\sharp\otimes Z_{s,t} V_{ J}^\sharp =V_I\otimes V_J\stackrel{(\ref{compas})}{=}V'_{r,s,t}V_{I\cup J}\\&=&V'_{r,s,t}Z_{r,t}V_{I\cup J}^\sharp,
\end{eqnarray*}
and the conclusion follows. The theorem is proved. 
\end{proof}
\begin{example}
Let $\left(\{X_{s,t}\}_{0<s<t}, \{\chi_{r,s,t}\}_{0<r<s<t} \right)$ and $\{\mu_{s,t}\}_{0<s<t}$ be as in Example \ref{jun17}, and $\{\varphi_{s,t}\}_{0<s<t}$ be the co-unit of the $C^*$-subproduct system  $\cali{A}_{\op{com}}=\left(\{C(X_{s,t})\}_{0<s<t}, \{\Delta_{r,s,t}\}_{0<r<s<t}\right)$ associated with $\{\mu_{s,t}\}_{0<s<t}$. Consider the family of Borel probability measures $\{\mu_{s,t}^\sharp\}_{0<s<t}$, constructed in Example \ref{jun17}, and the associated co-unit $\{\varphi_{s,t}^\sharp\}_{0<s<t}$ of the $C^*$-product system $\cali{A}_{\op{com}}^\sharp=(\{C(X_{s,t}^\sharp)\}_{0<s<t},\,\{\Delta_{r,s,t}^\sharp\}_{0<r<s<t})$. Then the Bhat-Mukherjee dilation of $\cali{H}_{\{\varphi_{s,t}\}}=\left(\{L^2(X_{s,t}, \mu_{s,t})\}_{0<s<t}, \{V_{r,s,t}\}_{0<r<s<t}\right)$ is isomorphic to the Tsirelson product system of Hilbert spaces $$\cali{H}_{\{\varphi_{s,t}^\sharp\}}=(\{L^2(X_{s,t}^\sharp, \mu_{s,t}^\sharp)\}_{0<s<t}, \{V_{r,s,t}'\}_{0<r<s<t}).$$
\end{example}

\end{document}